\documentclass[11pt]{amsart}
\usepackage{latexsym,amssymb}

\numberwithin{equation}{subsection}
\newcommand\NN{{\mathbb N}}
\newcommand\FF{{\mathbb F}}
\newcommand\QQ{{\mathbb Q}}

\newcommand\CC{{\mathbb C}}
\newcommand\ZZ{{\mathbb Z}}

\newcommand\AAA{{\mathcal A}}
\newcommand\FFF{{\mathcal F}}

\newcommand\Hilb{{\mathrm{Hilb}}}
\newcommand\Tor{{\mathrm{Tor}}}

\newcommand\Hom{{\mathrm{Hom}}}
\newcommand\hd{{\mathrm{hd}}}
\newcommand\depth{{\mathrm{depth}}}

\newcommand\rank{{\mathrm{rank}}}
\newcommand\Ann{{\mathrm{Ann}}}

\newcommand\Spec{{\mathrm{Spec}}}
\newcommand\mindeg{{\mathrm{mindeg}}}
\newcommand\Res{{\mathrm{Res}}}
\newcommand\Ind{{\mathrm{Ind}}}
\newcommand\Aut{{\mathrm{Aut}}}

\newcommand\mm{{\mathfrak m}}

\newcommand\triv{{\bf 1}}

\newtheorem{numberedproclaim}{Satz}[subsection]
\newtheorem{THM}[numberedproclaim]{Theorem}
\newtheorem{LEMMA}[numberedproclaim]{Lemma}
\newtheorem{PROP}[numberedproclaim]{Proposition}
\newtheorem{CORY}[numberedproclaim]{Corollary}

\let\LaTeXsection\section
\def\section{\setcounter{numberedproclaim}{0} \LaTeXsection}
\let\satzfont\rmfamily
\catcode`\@=11
\def\newSATZ#1{%
  \@ifnextchar[{\@osatz{#1}}{\@nsatz{#1}}}
\def\@nsatz#1#2{%
  \@ifnextchar[{\@xnsatz{#1}{#2}}{\@ynsatz{#1}{#2}}}
\def\@xnsatz#1#2[#3]{%
  \expandafter\@ifdefinable\csname #1\endcsname
    {\@definecounter{#1}\@newctr{#1}[#3]%
     \expandafter\xdef\csname the#1\endcsname{%
       \expandafter\noexpand\csname the#3\endcsname \@satzcountersep 
          \@satzcounter{#1}}%
     \global\@namedef{#1}{\@satz{#1}{#2}}%
     \global\@namedef{end#1}{\@endSATZ}}}
\def\@ynsatz#1#2{%
  \expandafter\@ifdefinable\csname #1\endcsname
    {\@definecounter{#1}%
     \expandafter\xdef\csname the#1\endcsname{\@satzcounter{#1}}%
     \global\@namedef{#1}{\@satz{#1}{#2}}%
     \global\@namedef{end#1}{\@endSATZ}}}
\def\@osatz#1[#2]#3{%
  \@ifundefined{c@#2}{\@nocounterr{#2}}%
    {\expandafter\@ifdefinable\csname #1\endcsname
    {\global\@namedef{the#1}{\@nameuse{the#2}}%
  \global\@namedef{#1}{\@satz{#2}{#3}}%
  \global\@namedef{end#1}{\@endSATZ}}}}
\def\@satz#1#2{%
  \refstepcounter{#1}%
  \@ifnextchar[{\@ysatz{#1}{#2}}{\@xsatz{#1}{#2}}}
\def\@xsatz#1#2{%
  \@beginSATZ{#2}{\csname the#1\endcsname}\ignorespaces}
\def\@ysatz#1#2[#3]{%
  \@opargbeginSATZ{#2}{\csname the#1\endcsname}{#3}\ignorespaces}
\def\@satzcounter#1{\noexpand\arabic{#1}}
\def\@satzcountersep{.}
\def\@beginSATZ#1#2{\trivlist 
   \item[\hskip \labelsep{\bfseries #1\ #2}]\satzfont}
\def\@opargbeginSATZ#1#2#3{\trivlist
      \item[\hskip \labelsep{\bfseries #1\ #2\ (#3)}]\satzfont}
\def\@endSATZ{\endtrivlist}
\catcode`\@=12
\newSATZ{EXAMPLE}[numberedproclaim]{Example}
\newtheorem{remark}[numberedproclaim]{Remark}
\def\REM{{\let\endtrivlist\leavevmode\begin{remark}\end{remark}}}
\newbox\shitbox
\long\def \unnumberedproclaim#1#2#3{\setbox\shitbox\hbox{\bf#2}%%
\vskip6pt
{\noindent\bf#1\ifdim\wd\shitbox>0pt(\unhbox\shitbox)\fi.\ }{\it #3}
\par\vskip6pt}
\def\nonumTHM#1 #2{\unnumberedproclaim{Theorem}{#1}{#2}}
\def\nonumCORY#1 #2{\unnumberedproclaim{Corollary}{#1}{#2}}
\def\nonumPROP#1 #2{\unnumberedproclaim{Proposition}{#1}{#2}}
\def\nonumLEMMA#1 #2{\unnumberedproclaim{Lemma}{#1}{#2}}
\catcode`\@=11
\newenvironment{proofof}[1][Missing Reference]{\par
  \normalfont
  \topsep6\p@\@plus6\p@ \trivlist
  \item[\hskip\labelsep\itshape
    \proofname~of #1\@addpunct{.}]\ignorespaces
}{%
  \qed\endtrivlist
}
\catcode`\@=12
\newcommand\Groth{{\bf R}}
\renewcommand\Theta{{\mathcal G}}
\def \Spec#1{{\mathrm{Spec}}(#1)}

\let\emphasis\bf
\def\bit{\bfseries\itshape}

\let\xref\LaTeXref
\let\xlabel\label
\let\sec\label
\let\fussnote\footnote
\let\sex\S

\let\iso\cong
\DeclareMathSymbol{\isomorphism}{\mathord}{symbols}{"27}
\def\lra{\mathop{\longrightarrow}}
\let\hra\hookrightarrow

\newcommand\gl{{\mathfrak l}}
\def\itemize{\begin{enumerate}\advance\itemindent by \parindent}
\catcode`\@=11
\newcommand\iitem{\renewcommand\labelenumi{\it(\@roman\c@enumi)}\item}
\newcommand\jitem{\renewcommand\labelenumi{(\@roman\c@enumi)}\item}
\catcode`\@=12
\def\enditemize{\end{enumerate}}
\def\altenditemize{\end{enumerate}\removelastskip}
\def\[{\relax\ifmmode\badmath\fi$$}
\def\]{\relax\ifmmode$$\else\badmath\fi\relax\ifmmode\badmath\fi\ignorespaces}
\def\badmath{\errmessage{Bad math environment: unbalanced {}s or $ signs parity error}}
\def\tridash{-\mkern-9mu-}
\def\res{\delimiter"526A33C }
\newcommand\tensor{\mathrel{\otimes}}
\renewcommand\mindeg{\mathrm{start}}
\long\def\comment#1{\ifvmode\removelastskip\else\unskip\fi}
\catcode`\@=11
\def\plaincases#1{\left\{\,\vcenter{\normalbaselines\m@th
    \ialign{$##\hfil$&\quad##\hfil\crcr#1\crcr}}\right.}
\def\plainmatrix#1{\null\,\vcenter{\normalbaselines\m@th
    \ialign{\hfil$##$\hfil&&\quad\hfil$##$\hfil\crcr
      \mathstrut\crcr\noalign{\kern-\baselineskip}
      #1\crcr\mathstrut\crcr\noalign{\kern-\baselineskip}}}\,}
\def\bmatrix#1{\left[\plainmatrix{#1}\right]}
\def\altplainmatrix#1{\null\,\vcenter{\normalbaselines\m@th
    \ialign{\hfil$##$\hfil&&\hfil$##$\hfil\crcr
      \mathstrut\crcr\noalign{\kern-\baselineskip}
      #1\crcr\mathstrut\crcr\noalign{\kern-\baselineskip}}}\,}
\newskip \gridskip
\def\grid#1{\null\,\vcenter{\normalbaselines\m@th
	\ialign{$##$\hfil&&\hskip\gridskip$##$\hfill\crcr
	\mathstrut\crcr\noalign{\kern-\baselineskip}
	#1\crcr\mathstrut\crcr\noalign{\kern-\baselineskip}}}\,}
\catcode`\@=12
\def\commadots{,\dots,}
\def \period{\mkern 1mu.}

\def \comma{\mkern 1.5mu,}

\catcode`\@=11
\def \fieldindent{\leavevmode\hbox to 2em{\hss}}
\def \field#1{\leavevmode\hbox to #1em{\hss}}
\def \fieldalign#1{\displ@y \tabskip\z@skip
\halign{\hbox to\displaywidth{$\@lign\fieldindent\displaystyle##\hfil$}\crcr
#1\crcr}}
\let \omega\zeta

\begin{document}

\title{Extending the Coinvariant Theorems of
Chevalley, Shephard--Todd, Mitchell, and Springer}

\author{Abraham Broer}
\address{D\'epartement de math\'emathiques et de statistique\\
Universit\'e de Montr\'eal\\
C.P. 6128, succursale Centre-ville\\
Montr\'eal (Qu\'ebec), Canada H3C 3J7}
\email{broera@DMS.UMontreal.CA}

\author{Victor  Reiner}
\address{School of Mathematics\\
University of Minnesota\\
Minneapolis, MN 55455, USA}
\email{reiner@math.umn.edu}

\author{Larry Smith}
\address{AG-Invariantentheorie\\
Mittelweg 3\\
D 37133 Friedland\\
Federal Republic of Germany}
\email{larry\underbar{\phantom{M}}smith@GMX.net}

\author{Peter Webb}
\address{School of Mathematics\\
University of Minnesota\\
Minneapolis, MN 55455, USA}
\email{webb@math.umn.edu}

\thanks{Work of first author supported by NSERC.
Work of second author supported by NSF grant DMS-0245379. Work of fourth author supported by MSRI}
%%  Any grant support that needs listing for Smith?? NO: I support myself.

\date{\today}


\begin{abstract}

We extend in several directions invariant theory results of 
Chevalley, Shephard and Todd, Mitchell and Springer.  Their results
compare the group algebra for a finite reflection group
with its coinvariant algebra, and compare a group representation
with its module of relative coinvariants. Our extensions apply to
arbitrary finite groups in any characteristic.
\end{abstract}

\maketitle

\tableofcontents
\catcode`\@=11
\def\@oddhead{\normalfont\scriptsize%%
	\hfil EXTENDING COINVARIANT THEOREMS \hfil\llap{\thepage}}
\catcode`\@=12

Let $k$ be an arbitrary field, $V$ an $n$-dimensional
$k$-vector space, and $G$ a finite subgroup of $GL(V)$\/.
Then $G$ acts by linear change of coordinates on the algebra $k[V]$ of
polynomial functions on $V$\/.
If $x_1 ,\ldots, x_n$ is a basis for the dual vector space
$V^*$ we may identify $k[V]$ with the polynomial algebra $k[x_1 ,\ldots, x_n]$
in $x_1 ,\ldots, x_n$ regarded as formal variables.
The {\emphasis coinvariant algebra} is the quotient algebra
\[
k[V]/(k[V]^G_+) \iso  k[V]\otimes_{k[V]^G} k
\]
where $(k[V]^G_+)$ denotes the ideal of $k[V]$ generated by the elements
of the invariant subalgebra $k[V]^G$ of strictly positive degree,
and $k=k[V]^G/k[V]^G_+$ is regarded as a trivial
$k[V]^G$-module. The coinvariant algebra is a finite dimensional
$G$-representation.
Much of its significance
(see, e.g., \cite{Benson}, \cite{Kane}, \cite{Smith})
derives from the fact that in favorable cases it provides a {\it graded}
version of the regular representation $k(G)$\/,
where $k(G)$ denotes the group algebra, but regarded as a $G$-representation.
This is made precise in the
following result, due to Chevalley (and also Shephard and Todd)
in characteristic zero, and to Mitchell in positive characteristic.

\nonumTHM{Chevalley \cite{Chevalley},
Shephard and Todd \cite{ShephardTodd}, Mitchell \cite{Mitchell}}
{Let $k$ be an arbitrary field, $V$ an $n$-dimensional
$k$-vector space, and $G$ a finite subgroup of $GL(V)$\/.
Suppose that $k[V]^G$ is a polynomial algebra. 
Then as $k(G)$-modules, the coinvariant algebra
$k[V]\otimes_{k[V]^G} k$ and the regular representation $k(G)$ have the same
composition factors counting multiplicities.}

Something similar for {\emphasis relative invariants} may be deduced
under suitable hypotheses as we explain next.
Suppose we have a second group $\Gamma$ and a $(k(\Gamma),k(G))$-bimodule
$U$, i.e., a right $k(G)$-module $U$ which has a commuting left action of
$k(\Gamma)$.
We chose to use the terminology of bimodules here as an aid in keeping
the different actions distinct.
If however no confusion can arise we frequently
identify $(k(\Gamma),k(G))$-bimodules with left
$k(\Gamma \times G)$-modules; the left action of
$(\gamma,g) \in \Gamma \times G$ on an element $u \in U$ being given by
$\gamma ug^{-1}$.
The module $M$ of $U$-{\emphasis relative invariants} is defined by
$M:=(U\otimes_k  k[V] )^G$ where $G$ acts on the
tensor product diagonally, viz., $g(u\otimes x)=ug^{-1}\otimes gx$.
Note that $M$ has the structure of a graded $(k(\Gamma),k[V]^G)$-bimodule,
in which $\Gamma$ acts trivially on $V$ and on $k[V]$.

	\nonumCORY {} {Let $k$ be a field, $V = k^n$ an $n$-dimensional
vector space over $k$\/, $G$ a finite subgroup of $GL(V)$, $\Gamma$ a second
finite group, and $U$ a finite-dimensional $(k(\Gamma),k(G))$-bimodule.
Regard the relative invariants
$
M=(U\otimes_k  k[V] )^G
$
as a $(k(\Gamma),k[V]^G)$-bimodule.
If $|G|$ is invertible in $k$ and $k[V]^G$ is a polynomial algebra,
then one has a $k(\Gamma)$-module isomorphism 
$
M\otimes_{k[V]^G}  k \cong U.
$
}

Note that this corollary includes the nonmodular version of the
Chevalley, Shephard-Todd, Mitchell Theorem
as the special case
where $\Gamma=G$ and $U=k(G)$ as a $(k(G),k(G))$-bimodule. 
More generally, if we let $H$ be any subgroup of $G$ and $\Gamma=N_G(H)$,
the normalizer of $H$ in $G$\/, we obtain a $k$-linear representation
$U = k (H\backslash G)$
from the permutation representation of $G$ on the set $H \backslash G$
of right cosets of $H$ in $G$\/.
Regarded as a $(k(\Gamma),k(G))$-bimodule, the relative invariants
$M=(U\otimes_k  k[V] )^G$ become the subalgebra of
$H$-invariant polynomials $k[V]^H$ regarded as a  $k(N_G(H))$-module
(see \sex \xref{induced-module-section} below).
Certain other cases of relative invariant
modules $M$ appear frequently in the literature, 
such as the $i^{th}$-exterior power
$U=\wedge^i(V^*)$ of $V^*$ 
(resp. $U=V$ itself),
and $M$ is the module of {\it $G$-invariant differential $i$-forms}
(resp. {\it $G$-invariant vector fields}) on $V$,
(see e.g., \cite{So}, \cite[\sex 6.1]{OrlikTerao}),
or, for a simple $k(G)$-module $U$ where the Hilbert series of
$M\otimes_{k[V]^G}  k$ defines the {\it fake degrees} for $U$
(see e.g., \cite[\sex 1.6]{GeckMalle}).
Our first main result, Theorem \xref{ChevalleyOmnibus} below,
extends the Chevalley,
Shephard--Todd, Mitchell result and its corollary
by removing the hypothesis that
$k[V]^G$ be polynomial and $|G|$ lie in $k^\times$\/.

Our other main results are inspired by a generalization of the
Chevalley-Shephard-Todd Theorem due to Springer \cite{Springer}, which incorporates
the action of an extra cyclic group.  We recall this next.
Given a finite subgroup $G \subseteq GL(V)$,
say that $v \in V$ is a {\emphasis regular vector} if 
the orbit $Gv$ is a {\emphasis regular orbit\/},
meaning that the stabilizer in $G$ of $v$ is 1, or equivalently that the orbit
achieves the maximum cardinality $|Gv|=|G|$. An element
$c \in G$ is a {\emphasis regular element} if it has a regular eigenvector 
$v \in V$, after possibly extending the field $k$ to include the corresponding 
eigenvalue $\omega \in k^\times$.  Letting $C=\langle c \rangle$ denote the
cyclic subgroup generated by $c$, the group algebra $k(G)$ becomes a
$(k(G),k(C))$-bimodule
in which $(g,c)$ acts on the basis element $t_h$ of $k(G)$
corresponding to $h$ in $G$ via $(g,c)  \cdot t_h:=t_{ghc}$.
As before we may identify $(k(G),k(C))$-bimodules with $k(G\times C)$-modules
by letting a left action of $c\in C$ correspond to a right action of $c^{-1}$.
We also let $C$ act on $k[V]$ by the
algebra automorphisms which are a scalar multiplication in each degree and
determined by requiring $c^j(x_i)=\omega^jx_i$ for $i=1,2,\ldots,n$.
In this way we obtain the structure of a $(k(G),k(C))$-bimodule on
$k[V]$ as well as $k[V] \otimes_{k[V]^G} k$.
The following theorem was proven by Springer in characteristic zero, 
and extended to arbitrary fields in \cite{RSW}.

	\nonumTHM {Springer \cite{Springer}, Reiner-Stanton-Webb \cite{RSW}}
{Let $k$ be a field, $V$ a finite-dimensional $k$-vector space,
$G$ a finite subgroup $G \subset GL(V)$.
Suppose that $k[V]^G$ is a polynomial algebra, and $c$ is a regular element of
$G$.
Let $C=\langle c \rangle$ as above.  Then one has  the equality
\[
\left[ k[V]\otimes_{k[V]^G} k \right] = [k(G)].
\]
in $\Groth(k(G \times C))$\/,
where $\Groth(k(G \times C))$ denotes the  Grothendieck ring of finite dimensional
$(k(G),k(C))$-bimodules.}

We are concerned with extensions of all these results to
arbitrary groups in any characteristic.
This will require significant reformulation since the naive versions
of these results would not be correct. For example, 
a simple consequence of the Chevalley, Shephard--Todd, Mitchell result is 
that the coinvariant algebra for $G$ has dimension $|G|$ whenever
$k[V]^G$ is a polynomial algebra, but this fails
when $k[V]^G$ is not polynomial (see e.g., \cite{LSsteintwo}).

\bigskip

The authors would like to thank Radha Kessar, Gennady Lyubeznik, Ezra Miller, Tonny Springer and Dennis Stanton
for helpful conversations and stimulating questions.

\section{Statement of Results}
\label{CST-intro-section}

In this section we state our main results and illustrate them with a
simple example.
For background material on invariant theory see
\cite{Benson, DerksenKemper, Smith}, on representation theory
see \cite{Sefive}, particularly Part III, and on reflection groups
see \cite{GeckMalle, Kane}.

\subsection{Chevalley, Shephard--Todd, Mitchell Type Results}
\label{chevie}

We first indicate how to remove the hypothesis that
$k[V]^G$ be polynomial, and $|G|$ lie in $k^\times$, from the
Chevalley, Shephard--Todd, Mitchell Theorem and its corollary.
Our result, Theorem \xref{ChevalleyOmnibus},
compares the ungraded $k(\Gamma)$-module $U$ with various graded
$k(\Gamma)$-modules, showing how, in a sense to be made precise below,
\itemize
\jitem $M\otimes_{k[V]^G}  k$ is an overestimate for $U$, 

\jitem it is the first in a sequence of alternating overestimates and
underestimates, and

\jitem these estimates converge to a suitably defined limit.
\altenditemize
To explain what this means, recall that $k(\Gamma)$ is not in general
semisimple, so
one compensates for this by working
with composition factors. A convenient way to do this
is to introduce the {\emphasis Grothendieck ring}
$\Groth(k(\Gamma))$ of finite-dimensional $k(\Gamma)$-modules. This is defined
to be the ring with one generator $[M]$ for each isomorphism class 
$\lbrace M \rbrace$ of finite dimensional $k(\Gamma)$-modules,
and one relation $[M'] - [M] + [M''] = 0$
for each short exact sequence
\[
0 \lra M' \lra M \lra M'' \lra 0
\period
\]
Addition in $\Groth(k(\Gamma))$ is induced by direct sum and product by tensor
product. There is also a partial ordering on $\Groth(k(\Gamma))$
defined by requiring
for two elements $x, y \in \Groth(k(\Gamma))$ that $x \geq y$ if
$x-y=[M]$ for some (genuine) $k(\Gamma)$-module $M$.  For example, given two
genuine modules $M_1$ and $M_2$, the inequality $[M_1] \geq [M_2]$ means that for
every simple $k(\Gamma)$-module $S$ one has the inequality
$[M_1:S] \geq [M_2:S]$, where $[M : S]$ denotes the multiplicity
of $S$ as a composition factor in $M$\/.

Given a finite group $\Gamma$, a (non-negatively) {\emphasis graded}
$k(\Gamma)$-module is one with a direct sum decomposition
$M=\oplus_{d \geq 0} M_d$
in which each $M_d$ is a finite-dimensional $k(\Gamma)$-module.
Such an $M$ gives rise to an element
$[M](t):=\sum_{d\ge 0} [M_d] t^d$ in the formal power series ring 
over the Grothendieck ring
$
\Groth(k(\Gamma))[[t]]:=\ZZ[[t]]\otimes_\ZZ \Groth(k(\Gamma)).
$
If one forgets the group action we obtain the formal power series
$\sum_{d \ge 0} \dim_k (M_d) t^d \in \ZZ[[t]]$ called the
{\emphasis Hilbert series} 
%(or {\emphasis Poincar\'e series})
of $M$ in this manuscript.

The motivating example for us of such a graded $k(\Gamma)$-module will be
\linebreak
$
\oplus_{i \geq 0} \Tor^{R}_i(M,k)
$
in the situation where $R$ is a finitely generated graded, connected,
commutative $k$-algebra with 
a grade-preserving action of $\Gamma$, and $M$ is a finitely generated graded
$R$-module with a compatible $k(\Gamma)$-module structure
(see Section \ref{rationality-section}).
These hypotheses imply that for each $i \geq 0$ the $k$-module
$\Tor^{R}_i(M, k)$ acquires a grading from $R$ and $M$,
and that each graded component $\Tor^{R}_i(M,k)_j$ is finite-dimensional over
$k$. Furthermore, for each fixed $i$ the $(i, j)$-component is non-zero for
only
finitely many $j$ and moreover all such $j$ are greater than or equal to $i$\/.
Likewise, for each fixed $j$ the $(i, j)$-component is non-zero for only
finitely many $i$.
This means for $i = 0 \comma 1 \comma\dots$ that
$\sum_{j \geq 0} (-1)^i [\Tor^{R}_i(M,k)_j]t^j$
is in fact a polynomial in $t$
which is is divisible by $t^i$ and, that the doubly infinite sum
\linebreak
$\sum_{i \geq 0} (-1)^i \sum_{j \geq 0} [\Tor^{R}_i (M, k)_j] t^i$
is a well defined element of $\Groth (k(\Gamma))[[t]]$\/.
This sum, which we often denote by
$\sum_{i \geq 0} (-1)^i [\Tor^{R}_i (M, k)](t)$\/,
generalizes the multiplicity symbol of Serre 
from the case with no group action (see e.g., \cite{Se} and \cite[\sex3]{smoke}).

\begin{THM}
\label{ChevalleyOmnibus}
{} {Let $k$ be a field, $V$ a finite-dimensional $k$-vector space,
$G$ a finite subgroup of $GL(V)$, and $\Gamma$ a finite group.
Let $U$ be a finite-dimensional
$(k(\Gamma),k(G))$-bimodule, and  let $M:= (U\otimes_k k[V] )^G$,
regarded as a $(k(\Gamma),k[V]^G)$-bimodule.  Set $K:=k(V)^G$,
the field of $G$-invariant rational functions on $V$
(i.e., the ungraded field of fractions of $k[V]^G$\/).
\itemize
\iitem In $\Groth(k(\Gamma))$, one has the inequality
\label{omnione}
\[
[ M\otimes_{k[V]^G} k ] \geq [U]. 
\]
Furthermore, one has equality if and only if $M$ is ${k[V]^G}$-free,
in which case $K \otimes_k U$ has a $K(\Gamma)$-module filtration $\{\FFF_j\}$
for which the factor $\FFF_j/\FFF_{j-1}$ is isomorphic to the
$j^{th}$ homogeneous component\linebreak[4]
$(M\otimes_{k[V]^G} k)_j\otimes_k K$.

\iitem More generally,
\label{omnitwo}
for any $m \geq 0$, in $\Groth(k(\Gamma))$, one has the inequality
\[
 \sum_{i=0}^m (-1)^i \sum_{j \geq 0} \left[ \Tor_i^{k[V]^G}(M,k)_j \right]
\plaincases
{
\geq [U] & if $m$ is even,\cr
\leq [U] & if $m$ is odd,\cr
}
\]
with equality if and only if $\Tor_i^{k[V]^G}(M,k)$ vanishes for $i > m$.

\iitem The element
\label{omnithree}
$$
\sum\limits_{i =0}^\infty (-1)^i \sum_{j=0}^\infty
	\left[ \Tor_i^{k[V]^G}(M,k)_j \right] t^i
$$
of $\Groth(k(\Gamma))[[t]]$ 
has the property that, for each simple module $S$, the power series in $t$ giving
the coefficient of $[S]$ lies in $\QQ(t)$.  Furthermore,
$t = 1$ is a regular value for these rational functions, and
the evaluation at $t=1$ is
\[
\sum_{i \geq 0} (-1)^i \sum_{j \geq 0}
\left[ \Tor_i^{k[V]^G}(M,k)_j \right]t^i \Bigg\res_{t=1}
= [U] \quad \text{ in }\Groth(k(\Gamma)).
\]
\altenditemize
}
\end{THM}

Theorem \xref{ChevalleyOmnibus}
is proven in Section \xref{integral-domain-section}, using a 
homological strengthening of Chevalley's method from \cite{Chevalley},
relying ultimately on the Normal Basis Theorem from Galois theory.
We illustrate this with a simple example.

\begin{EXAMPLE}
\xlabel{simpleexa}
{} Let $G$ be the cyclic group $\ZZ/2$ of order $2$
regarded as the subgroup of $GL(2,\CC)$ generated by the
scalar matrix $g=-I_{2 \times 2}$ which is the negative of the identity.
Note that the ring of invariants is not a polynomial algebra, rather one has
$$
R:=\CC[x, y]^{\ZZ/2}=\CC[x^2,xy,y^2].
$$
There are two simple $\CC(\ZZ/2)$-modules, $U_+$
and $U_-$, both $1$-dimensional, with $g$ acting
by the scalar $+1,-1$ on $U_+,U_-$, respectively.  

Choose $\Gamma$ to be the trivial group $\lbrace 1 \rbrace$ and regard it as
a subgroup of $\Aut_{\CC(\ZZ/2)}(U_{\pm})$.
The Grothendieck ring 
$\Groth(\CC(\Gamma))$ is isomorphic to $\ZZ$, with the isomorphism sending
the class $[\triv]$ of the trivial $1$-dimensional $\CC(\Gamma)$-module
to the integer $1$.  Any $\CC$-vector space of dimension $d$ then represents the 
element $[\CC^d]=d[\triv]$ in $\Groth(\CC(\Gamma))$.
In particular, $[U_+]=[U_-]=[\triv]$ in  $\Groth(\CC(\Gamma))$.

One can easily check that
\[\fieldalign
{
M_+:=(U_+ \otimes_\CC \CC[x, y] )^{\ZZ/2} = \CC[x, y]^{\ZZ/2} = R\cr
M_-:=(U_- \otimes_\CC \CC[x, y])^G
%= \CC[x, y]^{\ZZ/2} \cdot x + \CC[x, y]^{\ZZ/2} \cdot y 
= Rx + Ry.\cr
}
\]
So $M_+$ is
a free $R$-module of rank $1$, and all inequalities asserted
in
Theorem \xref{ChevalleyOmnibus} become trivial equalities.  
By contrast, $M_-$ has an interesting, infinite, $2$-periodic\fussnote{In fact,
whenever
$R$ is a {\it hypersurface algebra}, i.e.,
\[
R =k[V]^G \iso k[f_1,\ldots,f_n,f_{n+1}]/(h)
\]
for a single homogeneous relation $h$ among the
$f_i$'s, 
there will always be such an $R$-free resolution of
$M$ which is eventually $2$-periodic
(see e.g., \cite[\sex 6]{tate} or \cite{EisenbudCI}).}
$R$-free resolution
\[
\cdots \lra^{d_4} R(-7)^2 \lra^{d_3} R(-5)^2 
 	\lra^{d_2} R(-3)^2 \lra^{d_1} R(-1)^2
	\lra^{d_0} M_-  \rightarrow 0
\comma
\]
in which $R(-d)$ denotes a free $R$-module of rank $1$
having a basis element of degree $d$.  Here the differential
$d_0$ maps the two basis elements of $R(-1)^2$ onto $x, y$
in $M_-$, while the differentials $d_i$ for $i \geq 1$ can be chosen as
follows:
\[
d_i =
\bmatrix
{
x^2 & xy \cr
xy & y^2 
}
\text{ for even }i, \qquad 
d_i =
\bmatrix
{
y^2 & -xy \cr
-xy & x^2 
}
\text{ for odd }i.
\]
For any $m \geq 0$, this gives the following strict inequalities in
$\Groth(\CC(\Gamma)) \cong \ZZ$ 
\[\gridskip=0pt
\left[ \sum_{i=0}^m	
(-1)^i \sum_{j \geq 0} \Tor^R_i(M_-,\CC)_j \right] = 
\grid
{
[\CC^2] &= 2[\triv] & > [\triv]  &= [U_-] &~\mbox{if}~m~\mbox{is even,}\cr
[0] &= 0[\triv] & < [\triv]  &= [U_-] &~\mbox{if}~m~\mbox{is odd,}\cr
}
\]
as predicted by Theorem \xref{ChevalleyOmnibus}
(\romannumeral\xref{omnione}, \romannumeral\xref{omnitwo}).
In the limit as $m \rightarrow \infty$, one makes sense of this by
noting that
$
\Tor^{R}_i(M_-, \CC)_j
$ 
vanishes unless the internal degree $j$ and homological
degree $i$ satisfy $j=2i+1$.  Hence one can calculate 
in $\Groth(\CC(\Gamma))[[t]] \cong \ZZ[[t]]$ that 
\[
\begin{aligned}
\sum_{i \geq 0} (-1)^i \,\, \left[\Tor^{R}_i(M_-,\CC) \right](t)
	&= 2[\triv] t^1 - 2[\triv] t^3 + 2[\triv] t^5 - 2[\triv] t^7 +  \cdots\cr
&= \frac{2t}{1+t^2} [\triv].
\end{aligned}
\]
The coefficient of $[\triv] \in \Groth(\CC(\Gamma))$
is a rational function of $t$ with $t = 1$ as a regular value, and
upon substituting $t=1$, one obtains in $\Groth(\CC(\Gamma))$
\[
\sum_{i \geq 0} (-1)^i \left[\Tor^{R}_i(M_-,\CC) \right] (t)
\Bigg\res_{t=1}
= 
\frac{2t}{1+t^2} [\triv]\Bigg\res_{t=1} = \frac{2}{2} [\triv]
= [\triv] = [U_-]
\]
as predicted by  Theorem \xref{ChevalleyOmnibus}(\romannumeral3).
\end{EXAMPLE}

\subsection{Springer-Type Results}
\sec{SpringerIntroSection}

Let $k$ be a field, $V$ a finite-dimensional $k$-vector space,
$G$ a finite subgroup of $GL(V)$.
Suppose that $k[V]^G$ is a polynomial algebra, and $c$ is a regular element of
$G$.
As with the Chevalley, Shephard-Todd, Mitchell Theorem, 
we wish to deduce a more general version of Springer's Theorem
that applies to an arbitrary $(k(\Gamma),k(G))$-bimodule $U$ for
any finite group $\Gamma$. 
The  $k(\Gamma)$-module of relative invariants
$M:=(U\otimes_k k[V] )^G$ carries a commuting $C$-action in which $C$
acts on $k[V]$ by scalars as before,
and does nothing to the factor of $U$ in $U \otimes_k k[V]$. In this way
$M$ becomes a graded $(k(\Gamma), k(C))$-bimodule. It is also possible to
view $U$ itself as a $(k(\Gamma), k(C))$-bimodule in a different way,
namely with the action of $C$ coming as the restriction
of the action of $G$, of which $C$ is a subgroup. From the 
theorem of Springer,
it is not hard to deduce that if $k[V]^G$ is a polynomial algebra
and both $|G|, |\Gamma|$
lie in $k^\times$, then as $(k(\Gamma), k(C))$-bimodules,
\[
M\otimes_{k[V]^G} k \iso  U.
\]
Our next main result shows how to remove the hypothesis
that $|G|, |\Gamma|$ lie in $k^\times$.

	\begin{THM} 
\label{PolynomialInvariantsTheorem}
{} {Let $k$ be any field, $V$
a finite-dimensional $k$-vector space of $k$ and $G$ a finite subgroup
of $GL(V)$ with $k[V]^G$ a polynomial algebra.
Let $C=\langle c \rangle$ be the
cyclic subgroup generated by a regular element $c$ in $G$ with regular
eigenvalue $\omega$ and
$U$ a finite-dimensional $(k(\Gamma), k(G))$-bimodule for some finite group
$\Gamma$. 
Regard  $M:=(U \otimes_k k[V])^G$ as a $k[V]^G$-module 
and as a $(k(\Gamma), k(C))$-bimodule,
where $C$ acts on $k[V]$ via scalar multiplication determined by
$c^j(x_i)=\omega^jx_i$ for $i=1,2,\ldots,n$ and $x_1 \commadots x_n$ is a basis
for $V^*$\/.

Then 
\[
\sum_{i \geq 0} (-1)^i \sum_{j \geq 0}
	\left[\Tor_i^{k[V]^G}(M,k)_j \right] = [U].
\]
in $\Groth(k(\Gamma \times C))$; the sum being finite since $k[V]^G$ is a
polynomial algebra.
}
\end{THM}

Proving Theorem \xref{PolynomialInvariantsTheorem} (in particular, with
{\bit no} hypothesis on the field $k$) was one of our
original motivations. It easily implies our main application,
Corollary \xref{CyclicSievingCorollary} below,
which resolves in the affirmative both Conjecture 3 and Question
4 in \cite{RSW}.

	\begin{CORY} 
\label{CyclicSievingCorollary}
{} {Let $k$ be
a field, $V$ a finite-dimensional $k$-vector space over $k$
and $H \subset G$ two nested finite subgroups of $GL(V)$.
Assume $k[V]^G$ is a polynomial algebra, and let
$C=\langle c \rangle$ be the cyclic subgroup generated
by a regular element $c$ in $G$,
with eigenvalue $\omega$ on some regular vector in $V$ and
$\hat\omega \in \CC$ a complex lift of $\omega$.
Let $\Gamma:=N_G(H)/H$, where $N_G(H)$ denotes
the normalizer of $H$ in $G$.
Give $k[V]^H$ the structure of
a $(k(\Gamma),k(C))$-bimodule in which $c$ scales the
variables $x_1 \commadots x_n$ in $V^*$ by
$\omega$, and $\Gamma$ acts by linear substitutions.
Let $k(H \backslash G)$ be the 
$k$-vector space with basis the
right costs of $H$ in $G$ regarded
as a $(k(\Gamma),k(C))$-bimodule where
\[
\gamma\cdot Hg\cdot c:= \gamma Hg
c=H \gamma g c.
\]

Then in $\Groth(k(\Gamma \times C))$ one has the equality
\[
\sum_{i \geq 0}(-1)^i \sum_{j \geq 0}
	\left[ \Tor_i^{k[V]^G}(k[V]^H,k)_j \right]
= \left[ k(H \backslash G) \right].
\]
Ignoring the $\Gamma$-action, this implies that the quotient of Hilbert series
\[
X(t):=\frac{[k[V]^H](t)}{[k[V]^G](t)},
\]
is a polynomial in $t$,
and together with the $C$-action on the set $X=H \backslash G$,
gives a triple $(X,X(t),C)$ that exhibits the
{\bf cyclic sieving phenomenon} of \cite{RStantonWhite}:
Namely, for each element $c^j$ in $C$
the cardinality of the fixed point set $X^{c^j} \subset X$
is
given by evaluating $X(t)$ at the complex root-of-unity $\hat\omega^j$ of
the same multiplicative order as $c^j$. In other words,
$	
|X^{c^j}| = \left[ X(t) \right]_{t=\hat{\omega}^j}.
$	
}
\end{CORY}

\subsection{A Further Generalization}
\sec{MoreGeneralResult}

Theorem \xref{PolynomialInvariantsTheorem} follows from
the more general Theorem \xref{maintheorem},
which does not make the assumption that
$k[V]^G$ is polynomial and has
other applications. The guiding principle in this generalization is to replace
the condition that $k[V]^G$ be a polynomial algebra with the assumption it
contains a Noether normalization $R \subseteq k[V]^G$ fulfilling certain key
technical conditions. Recall that a {\emphasis Noether normalization}
for $k[V]^G$ is a polynomial subalgebra $R \subset k[V]^G$
over which $k[V]^G$ is a finitely generated module.

To state our result requires some preliminaries.
Let $V$
be a $(k(G),k(C))$-bimodule with $G$ and $C$ finite  groups, so that
$C$ acts on $k[V]$ and
$k[V]^G$.  
Suppose there exists a homogeneous
and $C$-stable Noether normalization $R \subset k[V]^G \subset
k[V]$, with the
additional properties
that the fiber $\Phi_v:=\phi^{-1}(\phi(v))$
over the point
$\phi(v)$ in the ramified covering $V \rightarrow \Spec R $
has free (but not necessarily transitive) $G$-action, and is stable under $C$.  
This means that
$C$
preserves the tower of inclusions 
\[
\mm_{\phi(v)} \subset R \subset k[V]^G \subset k[V],
\]
where $\mm_{\phi(v)}$ is the (generally inhomogeneous) maximal ideal
in $R$ corresponding to $\phi(v)$ in $\Spec R $.  
% Thus the only condition imposed on $C$ is that it preserve the (generally
% inhomogeneous) ideal $\gm_{\phi(v)}$\/, i.e., that $C$ preserves the fibre
% $\Phi_v$\/.
The quotient ring
$
k[V]/\mm_{\phi(v)} k[V]  =:A(\Phi_v) 
$ 
can be
thought of as the coordinate ring for the fiber $\Phi_v$ regarded
as a (possibly non-reduced) subscheme of
the affine space $V$.
This ring $A(\Phi_v)$ carries an interesting $(k(G),k(C))$-bimodule structure,
whose precise description we defer
until \sex\xref{incorporating-cyclic-action-section}
where the extra generality is exploited.

However, it is worth mentioning here
what this bimodule structure looks like under the hypotheses
of Theorem \xref{PolynomialInvariantsTheorem}, that is,
if $k[V]^G$ is polynomial, so that we can choose $R = k[V]^G$,
and where $c$ is a regular element of $G$ 
with eigenvalue $\omega$ on a regular eigenvector $v$.
In this special case, 
if we take for $C$ the group of scalar matrices in
$GL(V)$ generated by $\omega$ times the identity matrix,
then $A(\Phi_v) \cong k(G)$ carries the same $(k(G),k(C))$-bimodule structure
as was described on $k(G)$ in Springer's theorem (Theorem
\xref{PolynomialInvariantsTheorem}).

Back in the general setting, given a $(k(\Gamma),k(G))$-bimodule $U$,
one lets $C$ act trivially on
$U$ and $\Gamma$ act trivially on $k[V]$.  In this
way, the relative invariants $M:=(U \otimes_k k[V])^G$ carry the
structure of a graded $(k(\Gamma),k(C))$-bimodule, compatible with its
$R$-module structure.  Similarly $(U \otimes_k A(\Phi_v))^G$ carries the
structure of a $(k(\Gamma),k(C))$-bimodule in which $\Gamma$
acts only on the $U$ factor, and $C$ acts only on the $A(\Phi_v)$ factor.

\begin{THM} {}
\label{maintheorem}
{
Let $k$ be a field, $G, \Gamma$\/, and $C$ finite groups, and $V$ a finite
dimensional $(k(G), k(C))$-bimodule on which $G$ acts faithfully
(so $G \subset GL(V)$). Regard $V, k[V]$ and $k[V]^G$ as trivial
$k(\Gamma)$-modules.
Suppose there is a Noether normalization $R \subset k[V]^G$
that is stable under the action of $C$ on $k[V]$.
Suppose further that there is a vector $v$ in $V$ such that the fiber
$\Phi_v:=\phi^{-1}(\phi(v))$ containing $v$
for the map $\phi: V \rightarrow \Spec R$ 
carries both a free (but not necessarily transitive) $G$-action and is stable
under $C$.
Denote by $\mm_{\phi(v)}$ the maximal ideal in $R$
corresponding to $\phi(v)$ in $\Spec R$ and set
$A(\Phi_v) = k[V]/\mm_{\phi(v)} k[V]$ 
which is a $k(C)$-module.
Let $U$ be a finite-dim\-en\-sion\-al $(k(\Gamma),k(G))$-bimodule
regarded as a trivial $k(C)$-module.

Then the relative invariants $M:=(U \otimes_k k[V])^G$ satisfy
$$
\sum_{i \geq 0} (-1)^i \sum_{j \geq 0}
	\left[ \Tor_i^R( M, k )_j \right]
	= \left[ \left( U\otimes_k A(\Phi_v) \right)^G \right]
$$
in $\Groth(k (\Gamma \times C))$; the sum being finite since $R$ is a
polynomial algebra.
}
\end{THM}

From this we easily deduce
Theorem \xref{PolynomialInvariantsTheorem} 
in Section \xref{incorporating-cyclic-action-section}.
Section \xref{cyclic-only-section} illustrates a different type of
application of Theorem \xref{maintheorem}, namely
to Brauer character values in the simple group of order 168.

\section{Generalities for Rings, Modules, and $\Tor$}
\label{rationality-section}

We record here some issues surrounding $\Tor^R(M,k)$ and the action of a group 
on $R$-resolutions of $M$. Without the group action most of these results can
be found in \cite{smoke}.
Because the group varies in the applications we
denote it by the new symbol $\Theta$. It will always be assumed 
throughout Section~\ref{rationality-section} 
that
\itemize
\jitem \label{assumering}
$R=\oplus_{d \geq 0} R_d$ is a commutative,
$\NN$-graded, connected ($R_0=k$), Noetherian
$k$-algebra (i.e., finitely generated as an algebra over $k$) and
\jitem \label{assumegroup}
$\Theta$ is a group which acts on $R$ by graded $k$-algebra automorphisms.
\altenditemize
In addition we will consider finitely generated $R$-modules with a
compatible homogeneous action of $\Theta$ in
a sense
conveniently described in terms of  
the {\emphasis skew group algebra} $R\rtimes\Theta$
(see e.g., \cite{AuslanderReitenSmalo}).
This is the free $R$-module with elements $\{t_g\}_{g \in \Theta}$
indexed by $\Theta$ as a basis, and whose multiplication is determined by the rule 
$r t_g \cdot s t_h=r \cdot g(s)t_{gh}$ and bilinearity, for all
$r, s\in R$ and $g, h\in \Theta$.
We put a grading on $R\rtimes\Theta$ by requiring that an element
$r t_g$ has the same degree as $r$. 
An $R\rtimes\Theta$-module is the same thing as an $R$-module $M$
with an action of $\Theta$ on $M$ regarded as a graded abelian group
by grading preserving group endomorphisms satisfying $g(rm)=g(r)g(m)$ and 
$g_1(g_2m)=(g_1g_2)m$\/.
This is what we mean by a {\emphasis compatible} action of $\Theta$. 
So the third standing assumption in this section is that
\itemize
\catcode`\@=11\advance\c@enumi by 2\catcode`\@=12
\jitem \label{assumemodule}
$M=\oplus_{d \geq 0}M_d$ is an $\NN$-graded $R\rtimes\Theta$-module which is
Noetherian (i.e., finitely generated)  as an $R$-module.
\enditemize

The kind of graded $R\rtimes\Theta$-modules we will consider arises,
for example in the situation
discussed in the introduction: We have a $(k(G),k(C))$-bimodule $V$, so that
$C$ acts on $k[V], k[V]^G$ and possibly also on a Noether normalization
$R \subset k[V]^G$,
as well as compatibly on the $R$-module $M:=(U\otimes_k k[V] )^G$ for any
$(k(\Gamma),k(G))$-bimodule $U$.
In fact putting $\Theta=\Gamma\times C$, $M$ becomes an
$R\rtimes\Theta$-module.

\subsection{Review of Graded Resolutions}
\label{resolve}
Recall that,
ignoring group actions, 
there always exist
{\it graded $R$-free resolutions} $\FFF$
of $M$ in which all terms are finitely generated, that is, an exact sequence
$$
\cdots \overset{d_{i+1}}{\rightarrow}
F_i \overset{d_i}{\rightarrow}
F_{i-1} \overset{d_{i-1}}{\rightarrow} 
\cdots
F_1 \overset{d_1}{\rightarrow}
F_0  \overset{d_0}{\rightarrow} M \rightarrow 0
$$
with each $F_i$ a graded free $R$-module of finite rank,
and grade-preserving differentials $d_i$.  
From any such resolution one can compute the bigraded $k$-vector space
$\Tor^R(M,N) = \lbrace \Tor^R_i(M,N)_j\rbrace$ 
for any graded $R$-module $N$, 
by taking the homology of the tensored complex $\FFF \otimes_R N$.
This means that $\Tor_i^R(M,N)$ for $i \in \NN$ is a graded $k$-vector space;
the index $i$ is called the  {\emphasis homological grading},
and the grading on $\Tor^R_i(M, N)$\/,
namely the index $j$\/, is called the {\emphasis internal} grading.
Depending on the context we will use the notation
$\Tor^R(M, N)$
for the bigraded $\Tor$-functor, or its ungraded analog obtained by taking
the direct sum of the homogeneous components
$\Tor^R_i(M,N)_j$ for all $i$ and $j$\/.

It is possible to choose the resolution $\FFF$ to be {\emphasis minimal}
in the sense that
the ranks $\beta_i$ of the resolvents $F_i \cong R^{\beta_i}$ are
simultaneously all minimized;
this turns out to be equivalent to each differential $d_i$ having
entries in $R_+ = \oplus_{i > 0} R_i$.	
In particular, when $N=k=R/R_+$ is the trivial $R$-module,
if the complex is minimal then $\FFF \otimes_R N$ becomes a complex of
$k$-vector spaces with all zero differentials, showing that
$$
\beta_i=\dim_k \Tor^R_i(M,k).
$$

The length of a minimal resolution is called the 
{\emphasis homological dimension} $\hd_R(M)$,
that is, $\hd_R (M):=\min\{i: \Tor^R_i(M,k) \neq 0\}$.  Note that $\hd_R (M)$
need not be finite.  However, {\it Hilbert's syzygy theorem} asserts that
when $R$ is a polynomial algebra on $n$ generators,
one always has $\hd_R(M) \leq n$.

Given an $\NN$-graded $k$-vector space $U=\sum_{d \geq 0} U_d$,
let 
$$
\mindeg (U):=\min\{d: U_d \neq 0\}.
$$
The usual construction of a minimal
free $R$-resolution $\FFF$ of $M$
(see e.g., the proof of Proposition \xref{rationality-proposition} (i) to follow)
shows that it enjoys the property
$\mindeg (F_{i+1}) > \mindeg (F_i)$.
We will show in Section~\ref{non-domain-section} that a similar
property holds after incorporating a finite group action.

\subsection{The Group Action on $\Tor$}
\label{GroupActionsOnTor}

We start by pointing out that for $R\rtimes\Theta$-modules
$M$ and $N$ (where $\Theta$ is a group which acts on $R$ by graded $k$-algebra
automorphisms) there are diagonal actions of $R$ and $\Theta$
on $M\otimes_RN$ and also on $\Tor_i^R(M,N)$ making 
them into $R\rtimes\Theta$-modules.
Since $R$ is commutative, we allow ourselves 
to take the tensor product of two left $R$-modules.
We claim that for 
each $g\in\Theta$ the map $M\times N\to M\otimes_RN$ given by 
$(m,n)\mapsto g(m)\otimes g(n)$ is $R$-balanced. 
To establish this we must show that for each $r\in R$ we have 
$g(rm)\otimes g(n)= g(m)\otimes g(rn)$.
This is so because 
\[
g(rm)\otimes g(n) = g(r)g(m)\otimes g(n) = g(m)\otimes g(r)g(n)
	= g(m)\otimes g(rn).
\]
From this we obtain the diagonal action of $R\rtimes\Theta$ on $M\otimes_R N$.

In fact for each $g\in\Theta$ we have a natural transformation from the functor
\[-\otimes_RN:R \rtimes\Theta\hbox{-mod}\to R\hbox{-mod}\]
to itself, giving a functor
\[-\otimes_RN:R \rtimes\Theta\hbox{-mod}\to R\rtimes\Theta\hbox{-mod}.\]

We next show that the diagonal action of $\Theta$ extends to an action on
$\Tor_i^R(M,N)$ for $i > 0$. 
Regard $\Tor_i^R(-,N)$ as a functor
$R\rtimes\Theta\hbox{-mod}\to R\hbox{-mod}$\/.
For each $g\in\Theta$ and $i \in \NN$ we construct a natural transformation
$\eta_{g,i}$ from $\Tor_i^R(-,N)$ to itself so that these maps also commute
with the boundary maps in the 
long exact sequences which arise from 
any short exact sequence of 
$R \rtimes\Theta$-modules $0\to M_1\to M_2\to M_3\to 0$. 
For $i=0$ the natural transformation will be the one already constructed. 
To extend this for arbitrary $i$, we may take any complex of projective
$R \rtimes\Theta$-modules ${\mathcal P} = (\cdots\to P_2\to P_1\to P_0\to 0)$
which is acyclic except in degree zero where its homology is $M$.
Since $R\rtimes\Theta$ is free as an $R$-module this is also a projective
resolution of $M$ as an $R$-module,
so $\Tor_i^R(M,N)=H_i({\mathcal P}\otimes_R N)$.
As the action of $\Theta$ is by natural transformations of the functor
$-\otimes_R N$ it passes to an action on the complex
${\mathcal P}\otimes_R N$ and hence to an action on its homology. 
The verification that this action commutes with the boundary homomorphisms is a
standard argument in homological algebra.

We finally observe that if higher natural transformations exist with these
properties, they must be unique.
This follows from a homological degree shifting argument
(cf \cite[Chapter III]{CE}), 
since given any $R \rtimes\Theta$-module $M$ one may take a short exact
sequence $0\to K\to P\to M\to 0$ where $P$ is a projective
$R \rtimes\Theta$-module. This gives a long exact sequence
\[
0=\Tor_i^R(P,N)\to \Tor_i^R(M,N)\to 
	\Tor_{i-1}^R(K,N)\to \Tor_{i-1}^R(P,N)\to\cdots
\]
and the specification of $\eta_{g,i-1}$ on $\Tor_{i-1}^R(K,N)$ and
$\Tor_{i-1}^R(P,N)$ (which is only non-zero for $i=1$)
determine the specification of $\eta_{g,i}$ 
on $\Tor_i^R(M,N)$ since these maps commute with the connecting homomorphisms. 
Then the uniqueness implies that $\Tor_i^R(M,N)$ becomes an $R\rtimes\Theta$-module,
since the relations on the action of $\Theta$ which are needed for this hold 
if $i=0$, hence also for higher values of $i$.

\subsection{The Group Action on Resolutions}
\label{non-domain-section}

Given these preliminaries we can state and
prove our first rationality result.
This proposition can be regarded as an equivariant generalization of the
theorem of Hilbert--Serre (see e.g., \cite[Theorem 4.2]{smoke})
on the rationality of Poincar\'e series of
graded  Noetherian modules over Noetherian $k$-algebras.

\begin{PROP}
\label{rationality-proposition}
Let $R$ be a commutative $\NN$-graded Noetherian $k$ algebra,
$\Theta$ is a group which acts on $R$ by graded $k$-algebra automorphisms,
and $M$ a Noetherian $R \rtimes \Theta$-module.
\itemize
\iitem \label{existresol}	
There exists an $R$-resolution $\FFF=\{F_i\}_{i \geq 0}$ of $M$
in which each resolvent $F_i$ is not only a free $R$-module but also a
$k(\Theta)$-module, the maps are $k(\Theta)$-module morphisms, and
$\mindeg (F_{i+1}) > \mindeg (F_i)$.  
Consequently, the infinite 
sum
\[
\sum_{i \geq 0} (-1)^i [\Tor_i^R(M,k)](t)
\]
gives rise to a well-defined element in $\Groth(k(\Theta))[[t]]$.

\iitem \label{semisimplecase}	
When $k(\Theta)$ is semisimple (i.e., $|\Theta| \in k^\times$),
the resolution in (i) can in addition be chosen minimal as an $R$-resolution.

\iitem \label{hdfinite}
If $\hd_R (M)$ is finite, the resolution in
(\romannumeral\xref{existresol}) or in (\romannumeral\xref{semisimplecase})
can in addition be chosen with length $\hd_R (M)$.

\iitem \label{seriesinvertible}	
In $\Groth(k(\Theta))[[t]]$, the series $[R](t)$ is invertible,
and one has the following relation: 
\[
\sum_{i \geq 0} (-1)^i [\Tor_i^R(M,k)](t) = \frac{[M](t)}{[R](t)}.
\]

\iitem \label{allseries}
All three series
\[
[R](t)\comma~ [M](t)\comma~\sum_{i \geq 0} (-1)^i [\Tor_i^R(M,k)](t)
\]
in $\Groth(k(\Theta))[[t]]$
have the property that, for each simple $k(\Theta)$-module $S$,
the power series in $t$
counting the multiplicity of $[S]$ actually lies in $\QQ(t)$.

\iitem \label{multiplicity}
For each simple $k(\Theta)$-module $S$, the coefficient series
of $[S]$ in $[M](t)$ and $[R](t)$ have poles at $t=1$ of order at most
the Krull dimension of $R$.
\altenditemize
\end{PROP}

\begin{proof}
{\sf (\romannumeral\xref{existresol}):}
Because $M$ is finitely-generated as an $R$-module, the quotient
$M/R_+ M$ is a finite-dimensional, graded $k$-vector space.
By the graded version of Nakayama's Lemma
any homogeneous $k$-basis $\{\bar{m}_j\}$ for $M/R_+ M$ lifts to a
minimal homogeneous generating set $\{m_j\}$ for $M$ as an $R$-module.
Let $U$ be the $k(\Theta)$-submodule of $M$ generated by any choice of such
lifts $\{m_j\}$.
Then $U$ is a graded finite-dimensional subspace of $M$ because $\Theta$ is
finite and acts in a degree-preserving fashion on $M$.
Start a resolution $\FFF$ with the surjection
\[
F_0:= R \otimes_k U \overset{d_0}{\longrightarrow}  M\comma
\quad\mbox{where}\quad
r \otimes u  \longmapsto  ru\period
\]
Observe that the diagonal action of $k(\Theta)$ on $F_0:= R \otimes_k U$ is
required here, both to make $d_0$ a $k(\Theta)$-module morphism, and to make
the $R$-module structure on $F_0$ compatible with the $k(\Theta)$-module
structure of $R$.  
Observe also that $U\cong k\otimes_R F_0$, a relationship which will be used
in proving (\romannumeral\xref{seriesinvertible}).

Replacing $M$ by $\ker (d_0)$, we can iterate this process, and produce
the desired resolution $\FFF$, provided we can show that the inequality
$$
\mindeg (\ker (d_0)) > \mu:= \mindeg (M) =\mindeg (F_0)
$$
holds. However, this follows easily from the observation that
the restriction of the above map $d_0$ to the $\mu^{th}$ 
homogeneous components is the $k$-vector space isomorphism
\[
R_0 \otimes_k U_\mu =  k \otimes_k U_\mu \overset{d_0}{\longrightarrow}
	U_\mu = M_\mu\comma\quad\mbox{where}\quad
1 \otimes u  \longmapsto u\comma
\]
and hence $\ker (d_0)$ is nonzero only in degrees strictly larger than $\mu$.

\vskip .1in
\noindent
{\sf (\romannumeral\xref{semisimplecase}):}  If $k(\Theta)$ is semisimple,
then in the construction of $\FFF$ in (\romannumeral\xref{existresol}),
the $k(\Theta)$-submodule $U \subset M$ is a direct summand so there is
a $k(\Theta)$-module direct sum decomposition $M = U \oplus R_+M$.  Since
$U \cong M/R_+M$, iterating this construction will produce a minimal
resolution.

\vskip .1in
\noindent
{\sf (\romannumeral\xref{hdfinite}):} This we prove by induction on $\hd_R(M)$.
In the base case, i.e., where
$\hd_R (M)= 0$, $M$ is a free $R$-module and the assertion is trivial.
In the inductive step, note that after one step of the construction in
(\romannumeral\xref{existresol}) or (\romannumeral\xref{semisimplecase}),
there is a short exact sequence
\[
0 \rightarrow \ker (d_0) \rightarrow F_0 \rightarrow M \rightarrow 0
\]
whose long exact sequence in $\Tor^R(-,k)$ shows that
\[
\hd_R (\ker (d_0)) = \hd_R (M) - 1.
\]
Applying the inductive hypothesis to $\ker(d_0)$ gives the result.

\vskip .1in
\noindent
{\sf (\romannumeral\xref{seriesinvertible}):}
The fact that $[R](t)$ lies in $\Groth(k(\Theta))[[t]]^\times$
follows from the assumption that $R$ is a connected $k$-algebra 
with $\Theta$ acting trivially on $R_0=k$.  This means that the series expansion
of $[R](t)$ begins $[\triv] + [R_1] t^1 + [R_2] t^2 + \cdots$,
and $[\triv]$ is a unit of $\Groth(k(\Theta))$.

For the remaining assertion, start with a resolution $\FFF$ of $M$
produced as in (\romannumeral\xref{existresol}).  
There is the following string of equalities in $\Groth(k(\Theta))[[t]]$,
which are justified below.
\[
\begin{aligned}[]
[M](t) & = \sum_{i \geq 0} (-1)^i [F_i](t) \\
       & = \sum_{i \geq 0} (-1)^i [R \otimes_k (k \otimes_R F_i)](t) \\
       & = \sum_{i \geq 0} (-1)^i [R](t) \cdot [k \otimes_R F_i](t) \\
       & = [R](t) \sum_{i \geq 0} (-1)^i \cdot [k \otimes_R F_i](t) \\
       & = [R](t) \sum_{i \geq 0} (-1)^i \cdot [\Tor^R_i(M,k)](t) \\
\end{aligned}
\]
The first equality comes from looking at the Euler characteristic
for the (finite) exact sequence in each homogeneous component.

The second equality comes from the fact that $F_i$ is 
constructed as $R\otimes_k U_i$ and
$U_i\cong k\otimes_R F_i$, so $F_i\cong R\otimes_k (k\otimes_R F_i)$.

The third equality comes from 
the fact that in the isomorphism $F_i\cong R\otimes_k (k\otimes_R F_i)$
the action of $\Theta$ on the tensor product on the right is diagonal,
and this 
tensor product defines the product in
$\Groth(k(\Theta))[[t]]$.

The fourth equality is trivial.

The fifth equality holds because $\Tor^R(M,k)$ is the homology of the
complex $\FFF \otimes_R k$.
In each homogeneous component $(\FFF \otimes_R k)_d$
one has a finite complex of finite-dimensional $k$-vector spaces,
and the alternating sum $\FFF_i \tensor_R k$ over $i$
represents the same element in $\Groth(k(\Theta))$ as the
alternating sum of the $\Tor_i^R(M,k)_d$ in that component
(i.e., taking homology preserves Euler characteristics).

\vskip .1in
\noindent
{\sf  (\romannumeral\xref{allseries}):}
It suffices to prove the assertion for $[M](t)$;  one then takes
$M=R$ to deduce it for $[R](t)$, and uses
(\romannumeral\xref{seriesinvertible}) to deduce it for 
$\sum_{i \geq 0} (-1)^i [\Tor_i^R(M,k)](t)$\/.

To prove $[M](t)$ is invertible, one can reduce to the case where $R$ is
a polynomial algebra $A=k[f_1,\ldots,f_n]$ with trivial $\Theta$-action as
follows.  Note that since $R$ is Noetherian, by a result of 
Emmy Noether
$R^\Theta$ is also Noetherian.  Hence by the Noether Normalization Lemma,
$R^\Theta$ contains a homogeneous system of parameters $f_1,\ldots,f_n$, and we
put $A=k[f_1,\ldots,f_n]$.  The ring
extensions $A \hookrightarrow R^\Theta \hookrightarrow R$ are both integral
(i.e. module-finite),
and hence $M$ is also a finitely generated $A$-module.

In this case, one can apply (\romannumeral\xref{seriesinvertible}) to give
\begin{equation}
\label{M-in-terms-of-Tor-and-A}
 [M](t) =  
\left[\sum_{i \geq 0} (-1)^i [\Tor_i^R(M,k)](t)\right] \cdot [A](t).
\end{equation}
Note that 
the sum on the right is finite, i.,e., lies
in $\Groth(k(\Theta))[t]$, because Hilbert's syzygy theorem says
$\Tor^A(M,k)$ is finite dimensional.  Note also that 
because $\Theta$ acts trivially on $A$, one has
\[
[A](t) 
= \left( \prod_{i=1}^n \frac{1}{1-t^{\deg(f_i)}} \right)  [\triv],
\]
which is an element of $\QQ(t)$ times the class $[\triv]$ of the trivial
module. Hence 
for any simple module $S$ the series counting the multiplicity of $[S]$
 has $\QQ(t)$ coefficients.

\vskip .1in
\noindent
{\sf (\romannumeral\xref{multiplicity}):}
Again it suffices to prove it for $[M](t)$, and then take $M=R$ to deduce
it for $[R](t)$.  For $[M](t)$ it is implied by equation	
\eqref{M-in-terms-of-Tor-and-A} and the comments after it,
as after forgetting the group action 
the sum $\sum_{i \geq 0} (-1)^i [\Tor_i^R(M,k)](t)$
has $\ZZ[t]$ coefficients and the pole at $t=1$ in
$[A](t)$ is the Krull dimension of $A$.
This is the same as the Krull dimension of $R$ since
$A \hookrightarrow R$ is an integral extension.
\end{proof}

\subsection{A Short Review of Brauer theory}
\label{Brauer-review}

At several points we will need facts about the Grothendieck ring
$\Groth(k(\Gamma))$ for a finite group $\Gamma$ which can be conveniently
deduced from the theory of Brauer characters. We review this theory here
and refer to \cite[Part III]{Sefive} for details.

Let $p$ denote the characteristic of the ground field $k$\/.
We say that an element $\gamma$ in $\Gamma$ is {\emphasis $p$-regular} 

if its order lies in $k^\times$. Let $m$ be the least common multiple of the
orders of all $p$-regular elements
of $\Gamma$, and let $\xi$ be a primitive $m^{th}$ root of unity in some
extension field of $k$. Then
every $p$-regular element $\gamma$ in $\Gamma$ acting on a finite-dimensional
$k(G)$-module $U$ has all its eigenvalues among the group of 
$m^{th}$ roots of unity $\mu_m(k(\xi)^\times)=\langle\xi\rangle$.
Pick a primitive complex $m^{th}$ root of unity $\hat{\xi}\in\CC$ which will
lift $\xi$, and consider the resulting homomorphism which lifts
$m^{th}$ roots of unity from $k(\xi)$ to $\CC$:
$$
\begin{aligned}
\mu_m(k(\xi)^\times) & \overset{\rm lift}{\longrightarrow} \mu_m(\CC^\times)
\\
\xi^j        &\longmapsto \hat{\xi}^j.
\end{aligned}
$$
The {\emphasis Brauer character value} $\chi_U(\gamma)$ for $\gamma$
is then defined to be the sum of the lifts of the eigenvalues of $g$ on $U$.
Furthermore the field $k(\xi)$ is a splitting field for $\Gamma$ by a
theorem of Brauer (see e.g., \cite[\sex 12.3, Theorem 24]{Sefive}).

The Brauer character $\chi_U$ of a $k(\Gamma)$-module $U$ determines the
composition factors of $U$, and this 
has several important consequences. To begin with,
the collection of restriction homomorphisms
$$
\Groth(k(\Gamma)) \longrightarrow \Groth(k\langle \gamma \rangle),
$$
where $\gamma$ ranges over all $p$-regular elements
$\gamma$, determines elements of $\Groth(k(\Gamma))$ uniquely; that is, the map
$\Groth(k(\Gamma)) \rightarrow \bigoplus_\gamma \Groth(k\langle \gamma \rangle)$ is
injective.
It implies also that whenever one has a field extension $k \hookrightarrow K$, the
map 
$$
\Groth(k(\Gamma)) \overset{\psi_{k,K}}{\longrightarrow} \Groth(K(\Gamma))
$$
that is induced by extension of scalars $U \mapsto K \otimes_k U$ is injective, since
the Brauer character of a module remains the same after extending scalars.
So to prove an equality
in $\Groth(k(\Gamma))$ it
suffices to prove the equality in $\Groth(k\langle \gamma \rangle)$
for the $p$-regular elements $\gamma$ in $\Gamma$. 

If $\gamma\in\Gamma$ is a $p$-regular element then $k\langle \gamma \rangle$
is semisimple. Over a splitting field $k(\xi)$ the simple
$k\langle \gamma \rangle$-modules $U_j$ are all $1$-dimensional and are indexed
by $j \in \ZZ/d\ZZ$, where $d$ is the order of $\gamma$,
with $\gamma$ acting as the scalar $\omega^j$ for some primitive $d^{th}$ root
of unity $\omega$ in
$k^\times$.  An element in $\Groth(k\langle \gamma \rangle)$ is 
determined by the (virtual) composition multiplicities of each $U_j$.
If this element is of the form $[U]$ for some genuine $k(\Gamma)$-module $U$,
then $[U]$ will be determined by the dimensions
$\dim_k (U \otimes_k U_j)^{\langle \gamma \rangle}$
of its 
$U_j$-isotypic components.

Observe also that if one
has a commuting action of another finite group $\Gamma'$, and one
wants to prove an equality in $\Groth(k(\Gamma \times \Gamma'))$, it suffices
to prove it
in $\Groth((k\langle \gamma \rangle  \times \Gamma'))$ for each $p$-regular
$\gamma \in \Gamma$.
Furthermore, this can be done for genuine $k(\Gamma \times \Gamma')$-modules
by proving equality in $\Groth(k(\Gamma'))$
for each 
isotypic component.

\subsection{The Case of Domains with Trivial Group Action}
\label{integral-domain-section}

We assume the notations introduced
in 
Proposition \xref{rationality-proposition},
and in addition require that the
graded $k$-algebra $R$ be an integral domain on which $\Gamma$-acts trivially.
We let $K$ be the fraction field of $R$ and recall from the discussion of
Section~\ref{Brauer-review} that extension of scalars gives an inclusion of
Grothendieck rings 
$
\psi_{k,K}: \Groth(k(\Gamma)) \rightarrow \Groth(K(\Gamma))
$. We will make various assertions about isomorphism of $K(\Gamma)$-modules,
but where these modules are in fact defined over $k$ we may also deduce a
corresponding result for $k(\Gamma)$ modules(because $\psi_{k,K}$ is injective)
which we leave to the reader to formulate.

The main result in this section is an abstract version of Theorem
\xref{ChevalleyOmnibus}, 
from which Theorem \xref{ChevalleyOmnibus} will immediately follow. 
Before we state it, we present a lemma which will be needed in the proof.

\begin{LEMMA}
\label{field-change-lemma}
Let $M$ be a finitely generated graded $R(\Gamma)$-module, where $R$ is a
commutative graded Noetherian $k$-algebra on which $\Gamma$ acts
trivially. Assume that $R$ is an integral domain and let $R'\subseteq R$ be a
graded subring over which $R$ is integral. Let the fields of fractions of $R$
and $R'$ be $K$ and $K'$, respectively.
Then the map 
\[
\varphi:  K' \times M  \longrightarrow K \otimes_R M 
\]
sending $(a,m) \in K' \times M$ to $a \otimes_R m \in K \otimes_R M$
induces an isomorphism of $K'(\Gamma)$-modules 
\[
K' \otimes_{R'} M \rightarrow   K \otimes_R M\period
\]
Note that here $K$ is regarded as a $(K',R)$-bimodule, and
may even be regarded as
a $(K'(\Gamma),R(\Gamma))$-bimodule, on which $\Gamma$ acts trivially.
\end{LEMMA}

\begin{proof}
Note that the map $\varphi$ is $R'$-balanced simply because $R' \subset R$.
Hence it induces a well-defined map
$\varphi: K' \otimes_{R'} M \rightarrow   K \otimes_R M$,
which one can see is $K'$-linear and even a $K'(\Gamma)$-morphism.

It remains to show that $\varphi$ is a $K'$-vector space isomorphism, which is
facilitated by first observing that the $K'$-dimensions of the 
domain and range are the same, viz.,
\[
\begin{aligned}
\dim_{K'}  \left( K' \otimes_{R'} M \right) 
  &= \rank_{R'} (M) \\
  &= \rank_{R'} (R) \cdot \rank_R (M) \\ 
  &= [K:K'] \cdot \dim_K \left( K \otimes_R M \right) \\
  &= \dim_{K'} \left( K \otimes_R M \right)
\period
\end{aligned}
\]

Thus it suffices to show $\varphi$ is surjective.  For this one need
only check for any $s,r \in R$ with $r \neq 0$, and
any $m \in M$, that the decomposable tensor $\frac{s}{r} \otimes_R m$
is in the image of $\varphi$.  
For this we use that $R$ is integral over $R'$, so there is a
dependence 
\[
r^n + a_{n-1} r^{n-1} + \cdots +a_2 r^2 + a_1 r + a_0 = 0
\]
with $a_i \in R'$.  We may assume $a_0 \neq 0$, since one can divide by
$r$ in the domain $R$.  Let 
\[
m' := -(r^{n-1} + a_{n-1} r^{n-2} + \cdots a_2 r + a_1) m,
\]
which is an element of $M$, satisfying 
\[
r\cdot m' = a_0 \cdot m.
\]
Hence
\[
\frac{s}{r} \otimes_{R} m
=\frac{s}{ra_0} \otimes_{R} a_0 m
=\frac{s}{ra_0} \otimes_{R} rm'
=\frac{1}{a_0} \otimes_{R} sm' 
=\varphi\left(\frac{1}{a_0} \otimes_{R'} sm' \right)
\]
lies in the image of $\varphi$.
\end{proof}

\begin{THM}
\label{domain-with-trivial-action-theorem} Let $M$ be a finitely generated graded $R(\Gamma)$-module, where $R$ is a commutative graded connected Noetherian $k$-algebra on which $\Gamma$ acts trivially. Assume in addition that $R$ is an integral domain.
\itemize
\iitem \label{upperbound}
In $\Groth(K(\Gamma))$ we have
$$
[K\otimes_k (M\otimes_R k)]\geq  [K \otimes_R M].
$$
Furthermore, one has equality if and only if $M$ is $R$-free, in which case
the $K(\Gamma)$-module $K\otimes_R M$ has a filtration
\[
0=F_{-1}\subseteq F_0\subseteq F_1\subseteq\cdots\subseteq F_d=K\otimes_RM
\]
so that for each $j$ there is an isomorphism of $K(\Gamma)$-modules
\[
K \otimes_k  ( M \otimes_R k )_j\to F_j/F_{j-1}
\]
where $( M \otimes_R k )_j$ denote the $j^{th}$ homogeneous component of 
$M \otimes_R k $.

\iitem \label{altbound}
More generally, for any $m \geq 0$, in $\Groth(k(\Gamma))$, 
\[
 \sum_{i=0}^m(-1)^i \sum_{j \geq 0} \left[ K\otimes_k\Tor_i^R(M,k)_j \right] 
\plaincases
{
\geq  [K \otimes_R M] & if $m$ is even,\cr
\leq  [K \otimes_R M] & if $m$ is odd,\cr
}
\]
and equality holds if and only if $\hd_R(M) \leq m$, that is,
if and only if $\Tor_i^R(M,k)$ vanishes for $i > m$.

\iitem \label{limitcase}
Even if $\hd_R (M)$ is not finite, the formal power series
\[
\sum\limits_{i \geq 0} (-1)^i \,\, \left[ K \tensor_k \Tor_i^{R} (M,k) \right](t)
\]
of $\Groth(K(\Gamma))[[t]]$ 
has the property that,
for each simple $K(\Gamma)$-module $S$, the power series in $t$ giving
the coefficient of $[S]$ lies in $\QQ(t)$.
Moreover,
$t = 1$ is a regular value for these rational functions, and
\[
\sum_{i \geq 0} (-1)^i \,\, \left[ K \tensor_k \Tor_i^{R}(M,k) \right](t)\Big\res_{t=1}
= [K \tensor_R M] \quad \text{ in }\Groth(K(\Gamma)).
\]
\altenditemize
\end{THM}

Observe in the statement of this result that the action of $R$ on $K$ coming from the 
inclusion $R\subseteq K$ is not the same as the action coming from the 
composite homomorphism $R\to k\hookrightarrow K$, 
thus distinguishing $K\otimes_R M$ from $K \otimes_k  ( M\otimes_R k )$.

\begin{proof} We begin by proving (\romannumeral\xref{upperbound}).
Let $\{e_\alpha\}$ be a minimal homogeneous $R$-spanning subset for $M$.
Then $\{e_\alpha \otimes 1\}$ forms a $k$-basis for $M \otimes_R k$,
by the graded version of Nakayama's Lemma.  Hence 
$\{ 1 \otimes (e_\alpha  \otimes 1) \}$ forms a $K$-basis for 
$K \otimes_k \left( M \otimes_R k \right)$.

Also $\{ 1 \otimes e_\alpha\}$ is a $K$-spanning set 
for $K \otimes_R M$.  Filter $K \otimes_R M$ by letting $F_{j}$ be the
$K$-span of $1 \otimes e_\alpha$ for which $e_\alpha$ has degree at most $j$.
The module $K \otimes_k \left( M \otimes_R k \right)$
has a direct sum decomposition coming from its inherited grading, and there is
a composite mapping
\[
\altplainmatrix	
{
\left( K \otimes_k \left( M \otimes_R k \right) \right)_j 
         & \rightarrow & F_j               &\rightarrow   & F_j /F_{j-1} \cr
\cr
1 \otimes (e_\alpha  \otimes 1) 
         & \mapsto     & 1 \otimes e_\alpha & \mapsto 
		& \overline{1 \otimes e_\alpha}\cr
}
\leqno{\mbox{defined by}}
\]
where $e_\alpha$ is assumed to have degree exactly $j$. This composite mapping is a surjection of $K(\Gamma)$-modules.

These surjections show the inequality asserted in (i). One has equality
if and only if all these surjections are isomorphisms, that is, if and only if the 
$\{ 1 \otimes e_\alpha\}$ are a $K$-basis for $K \otimes_R M$, 
which happens if and only if they are  $K$-linearly independent.
This in turn happens if and only if the $\{e_\alpha\}$ are $R$-linearly independent, and hence
an $R$-basis for $M$, so that $M$ is free over $R$.

We turn next to the proof of (\romannumeral\xref{altbound}) and
(\romannumeral\xref{limitcase}).
In both of these proofs, it is convenient to reduce to the case where
$|\Gamma|$ lies in $k^\times$ and hence $k(\Gamma)$ is semisimple:
\comment{I find this trick amazingly clever!}
Recall from Section \ref{Brauer-review}
that virtual modules in $\Groth(k(\Gamma))$ are determined by their
restrictions to the cyclic subgroups
generated by $p$-regular elements $\gamma \in \Gamma$, so one may replace
$\Gamma$ with $\langle \gamma \rangle$ without loss of generality.

For the proof of (\romannumeral\xref{altbound}),
semisimplicity of $k(\Gamma)$ allows us to 
write down the first $m$ steps in a minimal $R$-free resolution of $M$ as in
Proposition \ref{rationality-proposition} (\romannumeral\xref{semisimplecase}),
so that all differentials are $k(\Gamma)$-module maps. Let $L$ denote the
kernel after the $m^{th}$ stage, so that one has the exact sequence
\[
0 \rightarrow L \rightarrow
F_m \rightarrow F_{m-1} \rightarrow \cdots
\rightarrow F_1 \rightarrow F_0 \rightarrow M \rightarrow 0.
\]
Applying the functor $K \otimes_R (-)$ is the same as a localization, and
hence gives rise to an exact sequence, whose $i^{th}$ term for
$i=0,1,\ldots,m$ looks like
\[
K \otimes_R F_i \,\, \cong \,\, K \otimes_R R \otimes_k (F_i/R_+F_i) \,\, 
	\cong \,\,  K \otimes_k \Tor_i^R(M,k)
\]
due to minimality of the resolution. This shows that
\[
(-1)^m \left(
\sum_{i=0}^m (-1)^i \,\, [ K \otimes_k \Tor^R_i(M,k) ] - [K \otimes_R M]
\right)
= [K \otimes_R L] \geq 0 
\]
in $K(\Gamma)$\/,
which gives the inequality in $\Groth(K(\Gamma))$ asserted in
(\romannumeral\xref{altbound}), with equality if and only if
$K \otimes_R L = 0$.  Since $L$ is a submodule of the free $R$-module $F_m$,
it is torsion-free
as an $R$-module and hence $K \otimes_R L = 0$ if and only if $L=0$.
Minimality of the resolution then 
shows that vanishing of $L$ (which is sometimes called the $(m+1)^{st}$
{\it syzygy module for} $M$) is equivalent to
$\hd_R(M) \leq m$.

For the proof of (\romannumeral\xref{limitcase}),
we note by Hilbert's syzygy theorem that it holds when $R$ is a polynomial algebra
as a special case of (\romannumeral\xref{altbound}).

To deduce the general case of (\romannumeral\xref{limitcase}),
extend the field $k$ if necessary in order to pick a graded Noether
normalization $R' \subseteq R$, that is a graded polynomial subalgebra $R'$
over which $R$ is module-finite, and let
$K' \subseteq K$ be the associated extension of fraction fields with degree
$[K:K']$.  We will take advantage of the injective ring homomorphisms from
Section~\ref{Brauer-review}
$$
\begin{aligned}
\Groth(k(\Gamma)) & \overset{\psi_{k,K'}}{\hookrightarrow}
    & \Groth(K'(\Gamma)) & \overset{\psi_{K',K}}{\hookrightarrow} & \Groth(K'(\Gamma))\\
\Groth(k(\Gamma))[[t]] & \overset{\psi_{k,K'}}{\hookrightarrow}
    & \Groth(K'(\Gamma))[[t]] & \overset{\psi_{K',K}}{\hookrightarrow} 
    & \Groth(K'(\Gamma))[[t]].\\
\end{aligned}
$$
that arise by extension of scalars in each case.
Applying Proposition \ref{rationality-proposition}
(\romannumeral\xref{seriesinvertible}) twice, one has in $\Groth(k(\Gamma))[[t]]$ that
$$
\begin{aligned}
\sum_{i \geq 0} (-1)^i [\Tor_i^R(M,k)](t) 
   & = \frac{ [M](t) }{ [R](t) } 
     = \frac{ [M](t) }{ [R'](t) } \cdot \frac{ [R'](t) }{ [R](t) }\\
   & = \left( \sum_{i \geq 0} (-1)^i [\Tor_i^{R'}(M,k)](t) \right) 
                \cdot \frac{ [R'](t) }{ [R](t) }.
\end{aligned}
$$
Applying the map $\psi_{k,K'}$, one concludes that in $\Groth(K'(\Gamma))[[t]]$ one has
\begin{equation}
\label{K'-power-series-equation}
\begin{aligned}
\sum_{i \geq 0} &(-1)^i [K' \tensor_k \Tor_i^R(M,k)](t) \\
   & = \left( \sum_{i \geq 0} (-1)^i [K' \tensor_k \Tor_i^{R'}(M,k)](t) \right) 
                \cdot \psi_{k,K'}\left( \frac{ [R'](t) }{ [R](t) }\right).
\end{aligned}
\end{equation}
The first factor on the right in 
\eqref{K'-power-series-equation} has $t=1$ as a regular value
in $\Groth(K'(\Gamma))$,
because $R'$ is a polynomial algebra, and the case already proven shows that 
the value taken there is $[K' \tensor_{R'} M]$.
For the second factor on the right in \eqref{K'-power-series-equation},
note that both $R', R$ carry trivial $\Gamma$-actions, 
and hence either \cite[Lemma 2.4.1 (iii) on page 20]{Benson} 
or \cite[Prop. 5.5.2 on page 123]{Smith} shows that this
factor also has $t=1$ as a regular value in $\Groth(K'(\Gamma))$,
taking the value $\frac{1}{[K:K']}[\triv]$. 
Consequently, the left side of \eqref{K'-power-series-equation} 
has $t=1$ as a regular value in $\Groth(K'(\Gamma))$ 
with evaluation
$$
\begin{aligned}
\sum_{i \geq 0}  (-1)^i [K' \tensor_k \Tor_i^R(M,k)](t) \Big\res_{t=1} 
    &= \frac{1}{[K:K']} [K' \tensor_{R'} M] \\
    &= \frac{1}{[K:K']} [K \tensor_{R} M]\comma \\
\end{aligned}
$$
where the second equality uses Lemma~\ref{field-change-lemma}.
Applying the map $\psi_{K',K}$, one concludes that 
the element 
$$
\sum_{i \geq 0} (-1)^i [K \tensor_k \Tor_i^R(M,k)](t)
$$
in $\Groth(K(\Gamma))[[t]]$ also has $t=1$ as a regular value in
$\Groth(K(\Gamma))$ 
and evaluating gives
$$
\begin{aligned}
\sum_{i \geq 0} (-1)^i [K \tensor_k \Tor_i^R(M,k)](t) \Big\res_{t=1} 
   & = \frac{1}{[K:K']} \psi_{K',K} [K \tensor_{R} M] \\
   & = \frac{1}{[K:K']} [K \tensor_{K'} K \tensor_{R} M] \\
   & = [K \tensor_{R} M].
\end{aligned}
$$
Here the last equality uses the fact that both $K', K$ carry
trivial $\Gamma$-actions, and $K\otimes_{K'}K\cong K^{[K:K']}$.
\end{proof}

\REM
Simple examples
show that the assumption that $\Gamma$ act trivially on $R$ in 
Theorem~\ref{domain-with-trivial-action-theorem}
cannot in general be removed.  
For example, consider the following inclusions of algebras.
\[
\underbrace{\CC[x^2, y^2]}_{\displaystyle R'}
\subset
\underbrace{\CC[x^2, xy, y^2] }_{\displaystyle R}
\subset
\underbrace{\CC[x, y]}_{\displaystyle M}
\]
Let $\Gamma=\ZZ/4\ZZ$ act compatibly on $R', R, M$ via the scalar substitution
given by $x, y \mapsto ix, iy$ (where $i^2=-1$ in $\CC$), 
so that their $d^{th}$ homogeneous components are scaled
by $i^d$.  Then $\Groth(k(\Gamma)) \cong \ZZ[\alpha]/(\alpha^4-1)$,
where $\alpha$ represents the class of the $1$-dimensional $k(\Gamma)$-module
that is scaled by $i$. In $\Groth(k(\Gamma))[[t]]$ there are equalities
\[
\fieldalign
{
\field{-1}
[\Tor^R(M,k)](t) = \frac{[M](t)}{[R](t)} \quad
	\text{( by Proposition~\ref{rationality-proposition}(iv) above)}\cr
 \field{1}
= \frac{[R'](t) [\Tor^{R'}(M,k)](t)} {[R'](t) [\Tor^{R'}(R,k)](t)} \quad 
\text{( again by Proposition~\ref{rationality-proposition}(iv) above)}\cr
\field{3}
= \frac{(1+\alpha t)^2 }{1+\alpha^2 t^2} \cr
\field{5}
= 1 + \frac{2\alpha t}{1+ \alpha^2 t^2} \cr
\field{7}
= 1 + \alpha \cdot \frac{2t}{1-t^4} + \alpha^2 \cdot 0
	+ \alpha^3 \cdot \frac{-2t^3}{1-t^4} \cr
}
\]
and the last of these does not have $t=1$ as a regular value.
Note that here $R$ is not fixed by $\Gamma$.

\subsection{Proof of Theorem~\ref{ChevalleyOmnibus}}
\label{proof-of-omnibus-theorem-section}
To prove Theorem~\ref{ChevalleyOmnibus}, one  applies
Theorem~\ref{domain-with-trivial-action-theorem} in the special case where
$R=k[V]^G$.  Then $K$ is its fraction field $k(V)^G$, and one has the
$R$-module $M=(U \otimes_k k[V])^G$, with action by $\Gamma$ coming from the
fact that $U$ is a $(k(\Gamma),k(G))$-bimodule.
What remains to show in this situation is that
\[
[ K \otimes_R M ] = [ K \otimes_k U ] \in \Groth(K(\Gamma))
\period
\]
In fact, these two $K(\Gamma)$-modules are {\it isomorphic}
as the following string of equalities and isomorphisms
proves.
$$
\begin{aligned}
K \otimes_R M 
&\overset{(1)}{=} k(V)^G \otimes_{k[V]^G} (U \otimes_k k[V])^G  \\
&\overset{(2)}{\cong} (U \otimes_k k(V))^G\\
&\overset{(3)}{\cong} (U\otimes_k KG)^G \\
&\overset{(4)}{\cong} K \otimes_k (U \otimes_k k(G))^G\\
&\overset{(5)}{\cong} K \otimes_k U.
\end{aligned}
$$
These may be justified as follows.
Equality $\overset{(1)}{=}$ is just a definition.
The isomorphism $\overset{(2)}{\cong}$ is given by the map
$$
\begin{aligned}
k(V)^G \otimes_{k[V]^G} (U \otimes_k k[V])^G &\rightarrow
	& (U \otimes_k k(V))^G\\
\frac{g}{h} \otimes \sum_i (u_i \otimes f_i) &\longmapsto
	& \sum_i u_i \otimes f
_i\frac{g}{h} .
\end{aligned}
$$
whose inverse sends
$$
\sum_i u_i\otimes  \frac{g_i}{h_i} 
=\sum_i u_i\otimes  \frac{\hat g_i}{h} 
\longmapsto 
\frac1h\otimes\left(\sum_i u_i\otimes \hat g_i\right)
$$
where $\hat g_i, h$ are chosen so that $\frac{\hat g_i}h = \frac{g_i}{h_i}$ 
and $h$ is $G$-invariant (e.g., choose $h$ to be the product of the finitely
many images $g(h_i)$ as $i$ varies and as $g$ varies through the finite
group $G$). The isomorphism $\overset{(3)}{\cong}$ comes from
the Normal Basis Theorem \cite[Theorem VIII.13.1]{Lang} applied to the Galois
extension
\[
K=k(V)^G \subset k(V),
\]
which asserts $k(V) \cong K(G)$ as $k(G)$-modules.
The isomorphism $\overset{(4)}{\cong}$ comes from the fact that $G$ acts
trivially on $K=k(V)^G$.
The isomorphism $\overset{(5)}{\cong}$ comes from the $k$-linear map defined by
$$
\begin{aligned}
U  &\rightarrow & (U\otimes_k  k(G))^G \\
u    &\longmapsto &  \sum_{g \in G} ug^{-1} \otimes t_g 
\end{aligned}
$$
where $t_g$ is the $k$-basis element in $k(G)$ indexed by $g$. 

This completes the proof of Theorem~\ref{ChevalleyOmnibus}.\qed

\subsection{Remarks on Group Cohomology and other related Constructions}
\label{higher-cohomology-section}

We reinterpret some of the foregoing results and comment
on their implications for the higher group cohomology
$H^j(G,U \otimes_k k[V])$.

Given a finitely generated graded integral domain $R$ over the field $k$
and a finite group $\Gamma$, we let
$\Gamma$ act trivially on $R$ and denote by
{\bf $R(\Gamma)$-mod} the abelian
category of finitely generated non-negatively graded $R(\Gamma)$-modules $M$.
The assignment $M \mapsto \Tor^R(M,k)$ is a functor from the category
{\bf $R(\Gamma)$-mod} to the category {\bf bigraded-$k(\Gamma)$-mod}
of bigraded-$k(\Gamma)$-modules that are finite-dimensional in each bidegree.
One can compose this with the forgetful functor
{\bf bigraded-$k(\Gamma)$-mod} $\rightarrow$ {\bf graded-$k(\Gamma)$-mod}
that forgets the internal grading by taking the direct sum over it,
and leaves the homological grading.
Theorem \ref{domain-with-trivial-action-theorem}(iii)
asserts that the composite of these two functors followed by taking the
alternating sum over the homological grading yields
a well-defined homomorphism of Grothendieck rings
$\Groth(R(\Gamma)) \rightarrow \Groth(k(\Gamma))$,
which sends $[M] \mapsto \psi^{-1}[K \otimes_R M]$, where $K$ is the field of
fractions of $R$ and $\psi$ is induced from the inclusion $k \hra K$\/.

Next we place ourselves in the context of
Theorem \ref{ChevalleyOmnibus} (\romannumeral\xref{omnithree})
where $R = k[V]^G$ and
we are considering
the {\it relative invariants} functor in the form
{\bf $k(\Gamma \times G)$-mod} $\rightarrow$ {\bf $R(\Gamma)$-mod}
that sends $U \mapsto (U \otimes_k k[V])^G$.  Then
Theorem~\ref{ChevalleyOmnibus} (\romannumeral\xref{omnithree})
says that if we follow this by the composite functor described above,
we get a well-defined homomorphism
$\Groth(k(\Gamma \times G)) \rightarrow \Groth(k(\Gamma))$
which sends $[U] \mapsto [U]$, that is, it coincides with the
restriction homomorphism that simply forgets the $k(G)$-structure.

This is a bit surprising, as the relative invariants functor 
is not exact;  it is the case $j=0$ of a family of (generally nontrivial) group
cohomology functors 
$$
U \mapsto M_j(U):=H^j(G, U\otimes_k k[V])
$$
for $j \geq 0$, which measure the inexactness of taking relative invariants.
In fact, it is not hard to show using the same ideas as above, that for any
strictly positive $j$, if one follows $M_j(-)$ by the composite functor and
takes an alternating sum (i.e., Euler characteristic) as 
discussed above, the result induces the {\it zero} homomorphism
$\Groth(k(\Gamma \times G)) \rightarrow \Groth(k(\Gamma))$.
That is, one has the following conclusion.

\begin{THM}
\label{higher-cohomology}
Let $G$ be a finite subgroup of $GL(V)$, let $R=k[V]^G$ and let $U$ be any
finite-dimensional $k(G \times \Gamma)$-module.
For all $j \geq 0$, let  
$$
M_j(U):=H^j(G, U\otimes_k k[V])
$$
be the $j^{th}$ cohomology group of $G$ with
coefficients in $U\otimes_k k[V]$, considered as a $R(\Gamma)$-module.
Then for strictly positive $j$ one has
$$
\sum_{i \geq 0} (-1)^i \,\, [\Tor_i^R(M_j(U),k)](t) \Bigg\res_{t=1} = 0
$$
in $\Groth(k(\Gamma))$
\end{THM}

Before proving this, we quote a standard fact about the behavior of group
cohomology under change of ground rings,	
which will also be of use further on in the proof of
Proposition \ref{completion-fixed-points-commute}.

\begin{PROP}
\label{exact-functors-commute-with-cohomology}
Let $G$ be a finite group and let $R\to S$ be a homomorphism of commutative
rings which is flat, so that $M\to S\otimes_R M$ is an exact functor from
$R$-modules to $S$-modules. Then for any $R(G)$-module $M$, one has
$$
S\otimes_R H^j(G,M)\cong H^j(G,S\otimes_R M)
$$ 
as $S$-modules.
\end{PROP}

\begin{proof}
See  \cite[Sec. 10.3, Prop. 7]{Gruenberg}.
\end{proof}

\begin{proofof}[Theorem~\ref{higher-cohomology}]
Applying Theorem~\ref{domain-with-trivial-action-theorem}(iii), and the remarks
in Section~\ref{Brauer-review} concerning extension of scalars, it suffices for
us to show that 
$$
K \otimes_R H^j(G,U\otimes_k k[V])=0 \text{ for }j > 0,
$$ 
where $K$ is
the field of fractions of $R$. Note that $K \otimes_R (-)$ is exact, because
it is a localization, so Proposition~\ref{exact-functors-commute-with-cohomology}
implies the first isomorphism in the following string
of isomorphisms and equalities.
$$
\begin{aligned}
K \otimes_R H^j(G,U\otimes_k k[V]) &\cong H^j(G, K \otimes_R (U \otimes_k k[V])) \\
  & \cong H^j(G, U \otimes_k k(V)) \\
  & \cong  H^j(G,U \otimes_k  KG) \\
  & \cong H^j(G, K \otimes_k (U \otimes_k k(G))) \\
  & = 0
\end{aligned}
$$
The equality on the second line is by definition, while the next three
isomorphisms appeared
as $\overset{(2)}{\cong}, \overset{(3)}{\cong}, \overset{(4)}{\cong}$
in Section~\ref{proof-of-omnibus-theorem-section}.
The last vanishing assertion comes from the
fact that $U \otimes_k k(G) $ is always free as a $k(G)$-module,
and hence $K \otimes_k (U \otimes_k k(G)))$ is free as a $K(G)$-module,
so its higher cohomology will vanish.
\end{proofof}

We close this section by noting
that whereas we have been dealing with the fixed points of
$G$ and its derived functors, we 
could instead have worked
with the fixed {\it quotients} by the
action of $G$. Specifically, the same arguments that we have used in the proof
of Theorem~\ref{ChevalleyOmnibus} can also be used if we replace $M$ by 
$M':= U\otimes_{k(G)}k[V]$, namely the fixed quotient of $G$ acting on 
$U\otimes_k k[V]$ rather than the fixed points which we have used before. Of course,
when $|G|$ is invertible in $k$ these constructions are isomorphic, and so they
are both valid interpretations of the notion of relative invariants in
the modular case.
With this definition of $M'$ the crucial chain of equations in
the proof of Theorem~\ref{ChevalleyOmnibus} becomes
$$
\begin{aligned}
K \otimes_R M'
&= k(V)^G \otimes_{k[V]^G} (U \otimes_{k(G)} k[V])  \\
&=U \otimes_{k(G)} (k[V]\otimes_{k[V]^G}k(V)^G)\\
&\cong (U \otimes_{k(G)} k(V))\\
&\cong (U\otimes_{k(G)} KG) \\
&\cong K \otimes_k (U \otimes_{k(G)} k(G))\\
&\cong K \otimes_k U
\end{aligned}
$$
and this shows that it is valid to replace $M$ by $M'$ in the statement of
Theorem~\ref{ChevalleyOmnibus}.

\section{Proof of Theorems~\ref{PolynomialInvariantsTheorem} and \ref{maintheorem}}
\label{scheme-section}

Let us recall the setting and 
statement of Theorem \ref{maintheorem}.
We suppose $k$ is an arbitrary field,
$G, \Gamma$\/, and $C$ finite groups, and $V$ a finite
dimensional $(k(G), k(C))$-bimodule on which $G$ acts faithfully,
so $G \subset GL(V)$. We regard $V, k[V]$ and $k[V]^G$ as trivial
$k(\Gamma)$-modules.
 
Suppose there is a Noether normalization $R \subset k[V]^G$
that is stable under the action of $C$ on $k[V]$.
Further suppose that one has a vector $v$ in $V$ such that the fiber
$\Phi_v:=\phi^{-1}(\phi(v))$ containing $v$ for the map 
$\phi: V \rightarrow \Spec R$ carries both
a free (but not necessarily transitive)
$G$-action and that this fiber $\Phi_v$ is stable under the action of $C$.
Denote by $\mm_{\phi(v)}$ the maximal ideal in $R$
corresponding to $\phi(v)$ in $\Spec R$ and introduce the coordinate ring
of $\Phi_v$ namely,
$A(\Phi_v) = k[V]/\mm_{\phi(v)} k[V]$,
which is a $k(C)$-module.

Let $U$ be a finite-dim\-en\-sion\-al $(k(\Gamma),k(G))$-bimodule
which we regard as a trivial $k(C)$-module.

In this situation Theorem \xref{maintheorem} asserts that
the relative invariants $M:=(U \otimes_k k[V])^G$ satisfy the equation
$$
\sum_{i \geq 0} (-1)^i \sum_{j \geq 0} \left[ \Tor_i^R( M, k )_j \right]
	= \left[ (U\otimes_k A(\Phi_v))^G \right].
$$
in $\Groth(k (\Gamma \times C))$. Note that Hilbert's syzygy theorem tells us
that the sum is finite since $R$ is a polynomial algebra.

In the subsections that follow we make various reductions leading up to the
proof of Theorem \ref{maintheorem}, with the goal of separating out the
different ideas involved.

\subsection{Reduction 1:  Removing the $\Gamma$-action}
\label{Brauer-character-section}

The intention here is to prove the following lemma.

\begin{LEMMA}
\label{reduction1-lemma}
Theorem~\ref{maintheorem} follows from its special case where
$\Gamma$ acts trivially on $U$; that is, the case with only a $C$-action
and no $\Gamma$-action.
\end{LEMMA}

This will essentially be a consequence of the Brauer theory reviewed in
Section~\ref{Brauer-review},  applied to the $k(\Gamma)$-modules
$(U \otimes_k k[V])^G$ and $(U\otimes_k A(\Phi_v))^G$.

Given a $p$-regular element $\gamma$ in $\Gamma$,
let $\tilde{G}:=\langle \gamma \rangle \times G$.
For any $(k(\Gamma),k(G))$-bimodule $W$, which we regard as a left
$k(\Gamma\times G)$-module, its
$G$-fixed subspace $W^G$ is a semisimple $k\langle \gamma \rangle$-module,
and one can express its $k\langle \gamma \rangle$-isotypic direct sum
decomposition in terms of $\tilde{G}$-fixed subspaces:
\begin{equation}
\label{isotypic-decomposition}
W^G \cong \bigoplus_{j \in \ZZ/d\ZZ}
	\left( W \otimes_k U_{j} \right)^{\tilde G}
\end{equation}
If $W$ happens also to have an action of $C$ commuting with the actions of
$\Gamma$ and $G$ this becomes a direct sum of $k(C)$-modules.

\begin{proofof}[Lemma \ref{reduction1-lemma}]
Apply 
the preceding discussion
to the $(k(G), k(\Gamma))$-bimodules
$W=U \otimes_k A(\Phi_v)$ and $W=U\otimes_k k[V]$.
Using the fact that functors like tensor product and $\Tor$ commute with
direct sums, along with the Brauer theory
from Section~\ref{Brauer-review}, one sees that Theorem~\ref{maintheorem} is
equivalent to showing for each $p$-regular element
$\gamma \in \Gamma$
\begin{equation}
\label{reduction1-equality}
 \sum_{i \geq 0} (-1)^i \sum_{j \geq 0}
\left[ \Tor^R( (\tilde{U}_s \otimes_k k[V])^{\tilde G}, k )_j \right] 
    = \left[ (\tilde{U}_s \otimes_k A(\Phi_v))^{\tilde{G}} \right]
\end{equation}
in $\Groth(k(C))$, where $U_s$ ranges over the simple
$k(\langle \gamma \rangle)$-modules and
$\tilde{U}_s := U \otimes_k U_s$. 
However, the equality \eqref{reduction1-equality} is an instance of
Theorem~\ref{maintheorem} with trivial
$\Gamma$-action: Simply replace $G$ by $\tilde{G}$, and
$U$ by $\tilde{U_s}$.
\end{proofof}

\subsection{Reduction 2:  Replacing Special Fiber with General Fiber}
\label{Borho-Kraft-section}

A key idea in the proof of Theorem~\ref{maintheorem} goes back to
Borho and Kraft \cite{BorhoKraft},
before that to Kostant \cite{Kostant}, and perhaps earlier: Given the finite 
ramified cover $\phi: V \rightarrow \Spec R$, one should compare the group
actions on the special fiber $\Phi_0$ and its coordinate ring $k[V]/R_+k[V]$ 
with the (often better understood) actions on a more general fiber 
$\Phi_v$ and its coordinate ring $k[V]/\mm_{\phi(v)}k[V]$.

To set this up, we return to the situation of Section~\ref{rationality-section}
with an $R$-module $M$ and compatible finite group $\Theta$ acting as above.
Then $k=R/R_+$ carries the trivial $R$-module and $k(\Theta)$-action.
Let $k'$ be a different, possibly non-graded, $R$-module structure on the field
$k$. In other words, $k'=R/\mm'$ where $\mm'$ is some (generally inhomogeneous)
maximal ideal of $R$ which happens to be $\Theta$-stable.  
Note that the action of $\Theta$ on $k'$ remains trivial, since $k'$ is 
spanned by the image of $1$, on which $\Theta$ acts trivially.

\begin{PROP}
\label{finite-hd-proposition}
Let $\Theta$ be a finite group,
$R$ an $\NN$-graded Noetherian algebra over the field $k$
and $M$ a finitely generated $R$-module. Assume that both $R$ and $M$ have
compatible $\Theta$-actions and that $\hd_R(M)$ is finite.
Denote by $\mm$ the (tautological) maximal ideal $R_+$ of $R$
(so $R /\mm \iso k$ is the tautological $R$-module structure on $k$), and
by $\mm'$ some (possibly inhomogenous)
maximal ideal of $R$ which is $\Theta$-stable. Set $k' = R/ \mm'$.

Then one has the following (ungraded) equality in $\Groth(k(\Theta))$
of two (finite) sums:
$$
\sum_{i, j \geq 0} (-1)^i [\Tor^R_i(M, k)_j]
=
\sum_{i, j \geq 0} (-1)^i [\Tor^R_i(M, k')_j].
$$
\end{PROP}
\begin{proof}
Compute either of the two $\Tor$'s by starting with a finite
(but not necessarily minimal) free $R$-resolution $\FFF$ of $M$ produced as in
Proposition~\ref{rationality-proposition}(iii),
tensoring over $R$ with $k$ or $k'$, and then taking the
homology of either $\FFF \otimes_R k$ or $\FFF \otimes_R k'$.  
Each term $F_i \otimes_R k$ or $F_i \otimes_R k'$
is a finite-dimensional $k$-vector space, and taking
alternating sums in $\Groth(k(\Theta))$
gives the following equalities
$$
\begin{aligned}
\sum_{i, j \geq 0} (-1)^i \left[\Tor^R_i(M,k)_j\right]
	&= \sum_{i \geq 0} (-1)^i [F_i \otimes_R k]\\
\sum_{i, j \geq 0} (-1)^i \left[\Tor^R_i(M,k')_j\right]
	&= \sum_{i \geq 0} (-1)^i [F_i \otimes_R k']
\period
\end{aligned}
$$
Note the sums are finite because $M$ is finitely generated and
$\hd_R(M)$ is finite. 

It therefore suffices to show that as $k(\Theta)$-modules only
(disregarding their $R$-module structure), one has
$
[F_i \otimes_R k] = [F_i \otimes_R k']
$
in $\Groth(k(\Theta))$, which we prove by a filtration argument. Given a homogeneous 
$R$-basis $\{e_\alpha\}$ for $F_i$, one has filtrations
$\AAA, \AAA'$ on the two $k$-vector spaces $F_i \otimes_R k, F_i \otimes_R k'$
defined as follows: let $\AAA_j, \AAA'_j$ be the $k$-span of those $k$-basis
elements $e_\alpha \otimes_R 1$ in which $\deg(e_\alpha) \leq j$.
Since $\Theta$ acts in a grade-preserving
fashion, it preserves these filtrations.  We claim that
there is also a $k(\Theta)$-module isomorphism 
$\AAA_j/\AAA_{j-1} \rightarrow \AAA'_j/\AAA'_{j-1}$
sending the $k$-basis element $e_\alpha \otimes 1$ to the $k$-basis
element $e_\alpha \otimes 1$.  To check that this isomorphism is $\Theta$-equivariant,
given $g \in \Theta$, let 
$
g(e_\alpha) = \sum_\beta r_{\beta,\alpha}(g) e_\beta
$ 
for some homogeneous elements $r_{\beta,\alpha}(g)$ in $R$. One then has, the
same computation in either of $\AAA_j/\AAA_{j-1}$ or $\AAA'_j/\AAA'_{j-1}$\/,
$$
g(e_\alpha \otimes 1) 
  = \sum_{\beta: \deg(e_\beta)=\deg(e_\alpha)} 
     r_{\beta,\alpha}(g) \left( e_\beta \otimes 1 \right)
\period
$$
Note that in this last sum, the coefficient $r_{\beta,\alpha}(g)$ in $R$ represents
the same element in the quotient fields $k$ or $k'$, since it must be of degree zero 
by homogeneity considerations.
\end{proof}

We comment that in the last part of the proof of Proposition 3.2.1 we do not 
necessarily get an isomorphism of $k(\Theta)$-modules, as can be seen by 
considering an example where $R=\FF_2[x]$ acted on trivially by a cyclic
group 
$\Theta$ of order 2. We may take $F_i=Re_0\oplus Re_1$ to be free of rank 2, 
where $e_0$ lies in degree 0 and $e_1$ lies in degree 1. Let $\Theta$ act on 
$F_i$ via the matrix
$$
\bmatrix{
1&x\cr
0&1\cr
}
$$
and let $k'=R/(x-1)$. Then $\Theta$ acts on $\mathcal{F}\otimes_R k$ via
$$
\bmatrix{
1&0\cr
0&1\cr
}
$$
and on $\mathcal{F}\otimes_R k'$ via
$$
\bmatrix{
1&1\cr
0&1\cr
}.
$$
These two actions are non-isomorphic.

\medskip

The particular case of Proposition~\ref{finite-hd-proposition} 
that will interest us most is where one has a finite group
$G \subseteq GL(V)$, with $R \subseteq k[V]^G$ an integral extension of graded
algebras.
Then given any finite-dimensional $k(G)$-module $U$, as in the introduction,
one can form the $R$-module $U \otimes_k k[V]$, with diagonal $k(G)$-action, 
having $G$-fixed subspace $M:=(U \otimes_k k[V])^G$ which
retains the structure of an $R$-module (because $G$ acts trivially on $R$).

\begin{PROP}
Let $G$ be a finite group that acts on the finite dimensional $k$-vector spaces
$V$ and $R \subset k[V]^G$ a Noether normalization. Then for any
finite dimensional  $k(G)$-module $U$ the module of relative invariants
$M:=(U\otimes_k k[V])^G$ is finitely-generated as
an $R$-module.
\end{PROP}
\begin{proof}
Recall the tower of integral extensions 
$$
R \subseteq k[V]^G \subseteq k[V].
$$
Note that $U \otimes_k k[V]$ is finitely-generated as a $k[V]$-module,
hence also finitely-generated as a $k[V]^G$-module.  Hence it is a Noetherian
$k[V]^G$-module, and its $k[V]^G$-submodule $M=(U \otimes_k k[V])^G$ 
will be Noetherian as well, so is finitely-generated over $k[V]^G$.  
But then $M$ is also finitely-generated over $R$.
\end{proof}

In the situation of Theorem~\ref{maintheorem},
with $R,M,\Gamma,C$ as defined there, choose
$$
k'=k_v:=R/\mm_{\phi(v)}.
$$
Note $\hd_R(M)$ is finite and bounded above by $\dim_k (V)$ via Hilbert's
syzygy theorem since $R$ is a polynomial algebra.
Thus Proposition~\ref{finite-hd-proposition} together
with Lemma~\ref{reduction1-lemma} show that to prove
Theorem~\ref{maintheorem} reduces to showing

\begin{equation}
\label{reduction2-equation}
\sum_{m \geq 0}(-1)^i \sum_{j \geq 0}
\left[ \Tor_i^R( M, k_v)_j \right] = \left[ (U\otimes_k A(\Phi_v))^G \right]
\end{equation}
in $\Groth(k(C))$.

The plan for proving \eqref{reduction2-equation}
is pretty clear:  Try to prove the vanishing result
\[
\Tor^R_i(M,k_v)=0~\mbox{for}~ i > 0\comma
\]
and then try to identify
\[
\Tor^R_0(M,k_v)  \cong (U \otimes_k A(\Phi_v))^G
\]
as $k(C)$-modules.

\subsection{Reduction 3:  Working Locally via Completions}
\label{completion-review-section}

Both parts in the aforementioned plan for proving \eqref{reduction2-equation}
in Section \ref{Borho-Kraft-section}
have the advantage that one can work with local and semi-local rings, by
taking completions with respect to $\mm_{\phi(v)}$-adic topologies.
Completion will be shown
to commute with $\Tor$, reducing the vanishing statement $\Tor^R_i(M,k_v)=0$
for $i > 0$ 
to showing that the completion $\hat{M}$ of $M$ is free as a module over the
completion $\hat{R}$ of $R$.  
Working locally also allows us to take advantage of the Chinese
Remainder Theorem for semi-local rings.

We begin with a review of some properties of completion in a general
setting. Much of this material can be found in \cite[Chap. 10]{AtiyahMacdonald},
\cite[Chap. 7]{Eisenbud}, \cite[\S 8]{Matsumura}.
Let $R$ be a commutative Noetherian ring, $M$ an $R$-module, and
$\gl$ any ideal of $R$.  Denote by $\hat{R}$  the {\emphasis completion} of
$R$ with respect to the $\gl$-adic topology, and
$\hat{M} \cong \hat{R} \otimes_R M$ the corresponding completion of $M$.

\begin{PROP}
\label{completion-fixed-points-commute}
Let $R$ be a commutative Noetherian ring, $M$ an $R$-module, $\gl$ an ideal of $R$,
and $\Theta$ a finite group acting on $M$ by $R$-module homomorphisms.
Then taking fixed points commutes with $\gl$-adic completion:
$
\left( \hat{M} \right)^\Theta  \cong \widehat{M^\Theta}
$
that is,
$$
(\hat{R} \otimes_R M)^\Theta  \cong \hat{R} \otimes_R M^\Theta
$$
as $\hat{R}$-modules.
\end{PROP}

\begin{proof}
Completion is an exact functor from $R$-modules to $\hat{R}$-modules,
so commutes with the cohomology functors $H^j(\Theta,-)$ by
Proposition~\ref{exact-functors-commute-with-cohomology}.
In particular, this holds for $j=0$, the fixed-point functor $(-)^\Theta$.
\end{proof}

If $\mm$ is a {\it maximal} ideal of $R$, then the ideal
$\hat\mm := \mm\hat{R}$ is also maximal in $\hat{R}$, and the two residue
fields are the same, viz.,
$$
k_\mm = R/\mm =\hat{R}/\hat\mm.
$$  
Thus $k_\mm$-vector spaces can be regarded as both $R$-modules and
$\hat{R}$-modules.  Furthermore, $\hat{R}$ is a {\it local} ring.  The next
proposition exploits this to translate
the vanishing of $\Tor_i$ for all $i > 0$ into freeness of the completed
module.

\begin{PROP}
\label{Tor-completion-commute}
Let $\mm$ be a maximal ideal of a commutative Noetherian ring $R$, and $M$ an $R$-module.
Then as $k_\mm$-vector spaces one has
$$
\Tor_i^{\hat{R}}(\hat{M},k_\mm) \cong \Tor_i^{R}(M,k_\mm).
$$
Consequently $\Tor_i^{R}(M,k_\mm)$ vanishes for all $i > 0$ if and only if $\hat{M}$ is a free $\hat{R}$-module.
\end{PROP}
\begin{proof}
The asserted isomorphism is a consequence of the string of isomorphisms that
follows.
$$
\begin{aligned}
\Tor_i^{\hat{R}}(\hat{M},k_\mm)
   & \cong \Tor_i^{\hat{R}}(\hat{R} \otimes_R \hat{M},\hat{R} \otimes_R k_\mm) \\
   & \cong \hat{R} \otimes_R \Tor_i^{R}(M,k_\mm) \\
   & \cong \Tor_i^{R}(M,k_\mm) \\
\end{aligned}
$$
The second of these uses the fact that $\hat{R} \otimes_R (-)$ is exact, and
the last
that $\Tor^R(M,N)$ is an $R/\Ann_R (N)$-module,
so a vector space over $k_\mm=R/\mm=\hat{R}/\hat\mm$,
and $\hat{R} \tensor_R k_\mm = k_\mm$\/.

Since $\hat{R}$ is a local ring with residue field $k_\mm$, the
$\hat{R}$-module $\hat{M}$ is free if and only if
$\Tor_i^{\hat{R}}(\hat{M},k_\mm) =0$ for $i > 0$.
So the previous isomorphism implies $\Tor^R_i(M, k_\mm) = 0$
for all $i > 0$\/.
\end{proof}

\subsection{(Semi-)local Analysis of the Fibers}
\label{vanishing-section}
Let $k$ be a field and $V$
a finite-dimensional $k$-vector space. Suppose that $G$ is a finite subgroup
of $GL(V)$ and $R \subset k[V]^G$ a Noether normalization,
so we have integral extensions of graded algebras
$$
R \subseteq k[V]^G \subseteq k[V]
\period
$$
Let $U$ be a finite-dimensional $(k(\Gamma), k(G))$-bimodule for some finite group
$\Gamma$  
with module of relative invariants $M:=(U \otimes_k k[V])^G$\/.
Suppose that there exists a vector $v \in V$ whose fiber $\Phi_v$
for the composite $\phi$ in the tower of ramified coverings
$$
V \rightarrow V/G \rightarrow \Spec R
$$
is permuted freely (but not necessarily transitively) by $G$.  
In this section we let $\Gamma$ act trivially on $V$\/ and
ignore all $C$-actions.  The $C$-actions will be put back in 
Section~\ref{incorporating-cyclic-action-section}.

For any $w\in V$, denote by $k_w$ the $k[V]$-module structure on $k$ defined by
$f\cdot \alpha:=f(w) \alpha$ for all $f\in k[V]$ and $\alpha \in k$.
The corresponding maximal ideal of polynomials vanishing at $w$ is denoted by 
$\mm_w\subset k[V]$.
By restriction we also consider $k_w$ as module over $k[V]^G, R$, etc.
The notations for the corresponding maximal ideals of $k[V]^G$ and $R$ are
$$
\begin{aligned}
\mm_{Gw}&=\mm_w\cap k[V]^G\\
\mm_{\phi(w)}&=\mm_w\cap R.
\end{aligned}
$$
The ideal $\mm_{Gw}$  can also be characterized as those $G$-invariant
polynomials that vanish at the $G$-orbit $Gw$ regarded as a point of
$V/G$. Likewise, $\mm_{\phi(w)}$
can be characterized as those polynomials in $R$ that vanish on the whole fiber
$$
\Phi_w:=\phi^{-1}(\phi(w))=\{u\in V: \phi(u)=\phi(w)\},
$$
and the generators are easy to describe:  If $R=k[f_1,\ldots,f_m]$ then
$$
\mm_{\phi(w)}=(f_1-f_1(w),f_2-f_2(w),\ldots,f_m-f_m(w))R.
$$
Let $\hat{R}$ denote the complete local ring obtained by completing $R$ at the
maximal ideal $\mm_{\phi(w)}$.

An important property of completion, the Chinese remainder theorem (see
e.g., \cite[Corollary 7.6]{Eisenbud}, \cite[Theorem 8.15]{Matsumura}), gives
cartesian product decompositions of complete semi-local rings.
Thus in the context just described
one has the following table in which the first column lists
various $R$-modules $M$, the second column lists their associated semi-local
completed $\hat{R}$-modules $\hat{M}:=\hat{R} \otimes_R M$ along with
Chinese remainder theorem isomorphisms, and the third column lists their
quotient  $k_w$-modules , viz.,
$
M/\mm_{\phi_(w)}M \cong \hat{M}/\hat\mm_{\phi(w)}\hat{M}.
$
\def \tablestrut{\vrule height 14pt depth 7pt width 0sp}
\[
\begin{tabular}{|c|c|c|}\hline
$R$-module $M$ & $\hat{R}$-module $\hat{M}$ & $k_w$-module $M/\mm_{\phi(w)}M$
\tablestrut\\ \hline\hline
$A:=k[V]$	& $\hat{A}  \cong \prod_{w \in \Phi_w} \hat{A}_w$
	& $A(\Phi_w)$
\tablestrut\\ \hline
$B:=k[V]^G$	& $\hat{B}(=\hat{A}^G) 
		\cong \prod_{Gw \in \Phi_w/G} \hat{B}_{Gw}$ & $B(\Phi_w/G)$
\tablestrut\\ \hline
$R$	& $\hat{R}$ & $k_w$ 
\tablestrut\\ \hline
\end{tabular}
\]
\centerline{{\bf Table \ref{vanishing-section}.1:} Chinese Remainder Table}
\def\ChineseRemainderTable{\ref{vanishing-section}.1}

\noindent
The heart of the matter now lies in using this to prove the following lemma.

\begin{LEMMA}
\label{regular-G-orbits}
Let $k$ be a field and $V$
a finite-dimensional $k$-vector space. Suppose that $G$ is a finite subgroup
of $GL(V)$ and $R \subset k[V]^G$ a Noether normalization.
Assume that there is a $v \in V$ whose fibre $\Phi_v$ carries a free $G$-action.
Let $U$ be a finite-dimensional $(k(\Gamma), k(G))$-bimodule for some finite group
$\Gamma$ with module of relative invariants $M:=(U \otimes_k k[V])^G$
considered as a $B(\Gamma)$-module
in the notations of Table \ChineseRemainderTable.
Then with the notations of that table the following hold.

\begin{enumerate}
\item[(i)] As $\hat{B}(\Gamma)$-modules, one has isomorphisms
\[
\hat{M} \cong (U \otimes_k \hat{B}(G))^G
        \cong \hat{B} \otimes_k (U\otimes_k k(G))^G.
\]
In particular, $\hat{A} \cong \hat{B}(G)$ as $\hat{B}(G)$-modules;
this is the special case $U=k(G)$.

\item[(ii)] $\hat{M}$ is a free $\hat{R}$-module, and hence
$$
\Tor^{\hat R}_i(\hat M,k_v)=0 \text{ for }i > 0.
$$

\item[(iii)] In $\Groth(k(\Gamma))$ there is an equality
\[
\sum_{i \geq 0} (-1)^i \sum_{j \geq 0} [\Tor^R_i(M, k)_j]
	= [\hat{M} \tensor_{\hat{R}} k_v]\period
\]

\item[(iv)] As $k(\Gamma)$-modules there are isomorphisms
\[
\hat{M} \tensor_{\hat{R}} k_v
    \cong ( U \otimes_k A(\Phi_v) )^G
	\cong (U \otimes_k B(\Phi_v/G)(G))^G.
\]
In particular, $A(\Phi_v) \cong B(\Phi_v/G)(G)$ as $B(\Phi_v/G)(G)$-modules;
this is the special case $U=k(G)$.
\end{enumerate}
\end{LEMMA}

\begin{proof}
We begin with some preparations making use of the isomorphism
$\hat{A} \cong \prod_{w \in \Phi_v} \hat{A}_w$
from Table \ChineseRemainderTable.
Here the right side has componentwise multiplication, and has a
$k(G)$-module structure given via the isomorphisms 
$\hat{A}_w \overset{g}{\rightarrow} \hat{A}_{gw}$
which permute the factors.
As a consequence of the assumption that $G$ acts freely on $\Phi_v$,
there is a decomposition of the fiber 
$$
\Phi_v=Gw_1 \sqcup \cdots \sqcup Gw_r
$$
into free $G$-orbits.  This gives two ways to regroup the factors, viz.,
$$
\hat{A} \cong \prod_{g \in G} 
  \underbrace{\left( \prod_{i=1}^r \hat{A}_{gw_i}\right)}_{:=\hat{A}_g}  
\qquad \text{ and } \qquad
\hat{A} \cong \prod_{i=1}^r 
  \underbrace{\left(
	\prod_{g \in G} \hat{A}_{gw_i}
		\right) }_{:=\hat{A}_{Gw_i}}.
$$
Note that here $\hat{A}_g = g(\hat{A_e})$ where 
$
\hat{A}_e=\prod_{i=1}^r \hat{A}_{w_i}.
$

View $\hat{A}$ as an $\hat{A}_e$-algebra and
$\hat{A}_{Gw_i}$ as a $\hat{A}_{w_i}$-algebra
via the diagonal embeddings
$$
\begin{array}{llll}
 \hat{A}_e     &\overset{\Delta}{\longrightarrow} &\prod_{g \in G} \hat{A}_g
	&\cong \hat{A} \\
 \hat{A}_{w_i} &\overset{\Delta}{\longrightarrow}
	&\prod_{g \in G} \hat{A}_{gw_i} &\cong \hat{A}_{Gw_i} \\
\end{array}
$$
defined by mapping $a$ to the product $(g(a))_{g \in G}$\/.
These embeddings give ring isomorphisms of $\hat{A}_e$ and $\hat{A}_{Gw_i}$
onto their diagonal images, which are exactly the $G$-invariant subrings of
the relevant 
completions, or the completions of the relevant $G$-invariant subrings by 
Proposition~\ref{completion-fixed-points-commute}. To wit,
\begin{equation}
\label{useful-ring-isomorphisms}
\begin{array}{llll}
\hat{A}_e     & \cong \Delta(\hat{A}_e)      & \cong \hat{A}^G
	& \cong \hat{B}\\
\hat{A}_{w_i} & \cong \Delta( \hat{A}_{w_i}) &\cong (\hat{A}_{Gw_i})^G
	&\cong \hat{B}_{Gw_i}
\end{array}
\end{equation}

Next consider the group algebras $\hat{A}_e(G)$ and $\hat{A}_{w_i}(G)$
for the group $G$,
with either $\hat{A}_e$ or $\hat{A}_{w_i}$ as coefficients;
in either case, denote by $t_g$ the
basis element corresponding to the element $g$ in $G$.
We claim that there are isomorphisms of
$\hat{A}_e(G)$-modules and $\hat{A}_{w_i}(G)$-modules
\begin{equation}
\label{group-algebra-isomorphisms}
\begin{array}{llll}
\hat{A}_e(G)
	&\overset{\alpha}{\longrightarrow}
		&\prod_{g \in G} \hat{A}_g
			&\cong \hat{A} \\
\hat{A}_{w_i}(G)
	&\overset{\alpha}{\longrightarrow}
		&\prod_{g \in G} \hat{A}_{gw_i}
			&\cong \hat{A}_{Gw_i} \\
\end{array}
\end{equation}
defined in both cases by  
$$
a t_g  \longmapsto     g(a) e_g
$$
where $e_g$ is the standard basis vector/idempotent corresponding to the
factor in the product indexed by $g$, and where 
$a$ lies either in $\hat{A}_e$ or $\hat{A}_{w_i}$
(so that $g(a)$ lies either in $g(\hat{A}_e)=\hat{A}_g$ or 
$g(\hat{A}_{w_i})=\hat{A}_{gw_i}$).
The inverse isomorphism $\alpha^{-1}$ in either case is defined by
$$
(a^{(g)})_{g \in G}=\sum_{g \in G} a^{(g)} e_g \quad
\longmapsto \quad
\sum_{g \in G} g^{-1} a^{(g)} t_g.
$$

With this preparation, we can start to prove the assertions of the Lemma,
beginning with assertion (i).
In light of the first isomorphism $\hat{A}_e \cong \hat{B}$ in
\eqref{useful-ring-isomorphisms}
the first isomorphism in
\eqref{group-algebra-isomorphisms}
shows $\hat{A} \cong \hat{B}(G)$ as a $\hat{B}(G)$-modules.
Since $G$ acts trivially on $R$, 
using Proposition~\ref{completion-fixed-points-commute} again, one has
$$ 
\begin{aligned}
\hat{M}:=\hat{R}\otimes_R ( (U\otimes_k A)^G )
 & \cong  (\hat{R}\otimes_R (U \otimes_k A))^G\\
 & \cong  (U \otimes_k \hat{A})^G \\
 & \cong  (U \otimes_k \hat{B}(G))^G \\
 & \cong  \hat{B} \otimes_k (U\otimes_k k(G))^G
\end{aligned}
$$
as $\hat{B}$-modules.
However, these are also $\hat{B}(\Gamma)$-module isomorphisms
since the $\Gamma$-action occurs entirely in the $U$ factor and acts trivially
on $R,B,A$ and their completions.

For (ii), it suffices to show that $\hat{M}$ is a free $\hat{R}$-module,
and then to apply Proposition~\ref{Tor-completion-commute}(ii).
Since (i) implies $\hat{M}$ is a free  $\hat{B}$-module
it suffices to show that  $\hat{B}$ is a free $\hat{R}$-module.
From Table \ChineseRemainderTable, one has
$\hat{B} \cong \prod_{i=1}^r \hat{B}_{Gw_i}$, and hence
it is enough to verify that each $\hat{B}_{Gw_i}$ is a free $\hat{R}$-module.
Note that the ring $\hat{B}_{Gw_i}$ is a finite extension of the ring
$\hat{R}$.
It turns out that both of these are regular local rings, because they
are isomorphic to completions of polynomial algebras at maximal ideals
(\cite[Corollary 19.14]{Eisenbud}):
in the case of $\hat{R}$ this is due to the assumption that
$R$ is polynomial, and in the case of $\hat{B}_{Gw_i}$ this is due to the
second isomorphism $\hat{B}_{Gw_i} \cong \hat{A}_{w_i}$
of \eqref{useful-ring-isomorphisms}.
Hence according to the Auslander-Buchsbaum Theorem
\cite[Theorem 19.9]{Eisenbud},
using $\dim (\tridash)$ to indicate Krull dimension of $\tridash$ , we obtain
$$
\begin{aligned}
\hd_{\hat{R}} (\hat{B}_{Gw_i}) 
 & = \dim (\hat{R}) - \depth_{\hat{R}} (\hat{B}_{Gw_i}) \\ 
 & = \dim (\hat{R}) - \dim (\hat{B}_{Gw_i}) \\
 & = \dim (\hat{R}) - \dim (\hat{R}) \\
 & = 0.
\end{aligned}
$$
The second equality here is due to the fact that regular local rings are
Cohen-Macaulay, so their depth and Krull dimension are the same,
while the third follows from
the fact that $\hat{B}_{Gw_i}$ is a finite extension of
$\hat{R}$.  Thus $\hat{B}_{Gw_i}$ is $\hat{R}$-free, and hence
so is $\hat{M}$.

For (iii), note that the sum is finite by Hilbert's Syzygy Theorem.
In $\Groth(k\Gamma)$, one then has the following equalities, justified below:
\[
\begin{aligned}
\sum_{i \geq 0} (-1)^i \sum_{j \geq 0} [\Tor^R_i(M, k)_j]
	& =
\sum_{i \geq 0} (-1)^i \sum_{j \geq 0} [\Tor^R_i(M, k_v)_j] \\
	& =
\sum_{i \geq 0} (-1)^i \sum_{j \geq 0} [\Tor^{\hat{R}}_i(\hat{M}, k_v)_j] \\
	& =
[\Tor^{\hat{R}}_0(\hat{M}, k_v)] \\
	& =
[\hat{M} \tensor_{\hat{R}} k_v].
\end{aligned}
\]
The first equality comes from applying Proposition~\ref{finite-hd-proposition} 
with $k'=k_v$, the second from Proposition \xref{Tor-completion-commute} with
$k_\mm=k_v$, the third from assertion (ii) above, and the last from the definition of
$\Tor_0$.

For (iv), note that 
$$
\begin{aligned}
\hat{M} \tensor_{\hat{R}} k_v
 &\cong  \hat{M}/ \hat{\mm}_{\phi(v)}\hat{M} \\
 &\cong   \left(U \otimes_k  \hat{A})^G \right)/ \hat{\mm}_{\phi(v)}
            \left( U \otimes_k \hat{A})^G\right)\\
  &\cong   \left(U
              \otimes_k  \left( \hat{A}/\hat{\mm}_{\phi(v)}\hat{A} \right) \right)^G \\
  &\cong   \left(U \otimes_k  \left( A/\hat{\mm}_{\phi(v)}A \right) \right)^G =:  (U \otimes_k A(\Phi_v))^G.\\
\end{aligned}
$$
This gives the first isomorphism in (iv).  For the second, note that
$$
\begin{aligned}
\hat{A}/\hat{\mm}_{\phi(v)}\hat{A}
 & \cong \hat{B}(G)/\hat{\mm}_{\phi(v)} \hat{B}(G)  \\
 & \cong (\hat{B}/\hat{\mm}_{\phi(v)} \hat{B})(G)  \\
 & \cong (B/\hat{\mm}_{\phi(v)}B)(G)\\
 & =  B(\Phi_v/G)(G)
\end{aligned}
$$
\end{proof}

\subsection{Finishing the proofs: 
Incorporating the $C$-Action}
\label{incorporating-cyclic-action-section}

Assertions (ii) and (iii) of Lemma~\ref{regular-G-orbits} complete the proof of Theorem~\ref{maintheorem},
except that we have
not yet accounted for the $C$-action.  We rectify this here, and explain
how Theorem~\ref{maintheorem} implies
Theorem~\ref{PolynomialInvariantsTheorem}.

So we assume there is also a finite subgroup $C \subset GL(V)$,
commuting with $G$, that preserves $R$ and the maximal ideal $\mm_{\phi(v)}$, 
making $C$ act on the fiber $\Phi_v$.  One then has compatible $C$-actions on
$$
\begin{array}{lllll}
R &\subseteq &B=k[V]^G &\subseteq &A=k[V]\\
\hat{R} & \subseteq &\hat{B} &\subseteq &\hat{A}\\
k_v &\subseteq & B(\Phi_v/G) &\subseteq& A(\Phi_v).
\end{array}
$$
We wish to describe these actions more explicitly under the assumption that
the $G$-orbits $\{Gw_i\}_{i=1}^r$ in $\Phi_v/G$ are all regular.
Note that $C$ permutes
these $G$-orbits, since it commutes with $G$, and if $c \in C$ stabilizes
some $G$-orbit $Gw_i$, then there will be a unique element $g_{c,w_i} \in G$
for which 
\begin{equation}
\label{orbit-rep-action}
cw_i = g_{c,w_i}w_i.
\end{equation}
One checks that this element $g_{c,w_i}$ depends
only on the choice of the representative
$w_i$ for the orbit $Gw_i$ up to conjugacy as follows. First
$$
g_{c,hw_i}=h g_{c,w_i} h^{-1}.
$$
However, once a choice of representative $w_i$ is made, one has
$$
cgw_i = g c w_i = g g_{c,w_i} w_i
$$ 
for all $g \in G$.

Also recall that for each $i=1,2,\ldots,r$, there is an isomorphism 
(see \eqref{useful-ring-isomorphisms} of \sex\ref{vanishing-section})
of the completed local rings $\hat{A}_{w_i} \cong \hat{B}_{Gw_i}$,
which are finite extensions of the local ring $\hat{R}_{\mm_{\phi(v)}}$.  
Let 
$$
\begin{aligned}
B_{(Gw_i)}&:= B_{Gw_i}/\mm_{\phi(v)}B_{Gw_i} \\
       &\cong \hat{B}_{Gw_i}/\hat\mm_{\phi(v)}\hat{B}_{Gw_i}\\
       &\cong \hat{A}_{w_i}/\hat\mm_{\phi(v)}\hat{A}_{w_i}\\
       &\cong A_{w_i}/\mm_{\phi(v)}A(w_i). \\
\end{aligned}
$$
This ring is a finite-dimensional $k_v$-vector space;
it may be viewed either as the coordinate
ring for the (possibly non-reduced) structure on the fiber
$\Phi_v$ local to the point $w_i$, as a subscheme of $V$, or for the structure
on the orbit space $\Phi_v/G$ local to the orbit $Gw_i$, as a subscheme of
$V/G$.

\begin{LEMMA}
\label{add-in-cyclic-action}
Assume the notation and hypotheses of Lemma~\ref{regular-G-orbits}.
Further assume, as in this section, that there is a finite subgroup $C \subset GL(V)$
which preserves the Noether normalization $R$ and the fiber $\Phi_v$,
Choose for each $G$-orbit $Gw_i$ within $\Phi_v$ a representative $w_i$,
and define $g_{c,w_i}$ in $G$ via \eqref{orbit-rep-action} whenever
$cGw_i=Gw_i$.

Then the
isomorphism of $B(\Phi_v/G)(G)$-module asserted in part (iv)
of that lemma, viz.,
\[
\hat{M} \tensor_{\hat{R}} k_v
     \cong (U \otimes_k  A(\Phi_v) )^G
		\cong (U\otimes_k B(\Phi_v/G)(G) )^G.
\]
is also a $k(C)$-module isomorphism, 
obtained by using the $k(C)$-module structure
induced from the following isomorphisms:
\[
A(\Phi_v) \cong B(\Phi_v/G)(G) \cong  \prod_{i=1}^r B_{Gw_i}(G).
\]

If $c \in C$ has $cGw_i=Gw_j$ for $j \neq i$, then there is
a ring isomorphism $B_{Gw_i} \overset{c}{\rightarrow} B_{Gw_j}$
that extends to a ring isomorphism
$B_{Gw_i}(G) \overset{c}{\rightarrow} B_{Gw_j}(G)$.

For $c \in C$ with $cGw_i=Gw_i$, the ring automorphism 
$B_{Gw_i} \overset{c}{\rightarrow} B_{Gw_i}$ extends to a ring automorphism   
$$
\begin{aligned}
B_{Gw_i}(G) & \overset{c}{\longrightarrow} & B_{Gw_i}(G) \\
a t_g & \longmapsto & c(a) t_{g g_{c,w_i} }.
\end{aligned}
$$

Consequently, there is the following identity relating Brauer character values:

\begin{equation}
\label{Brauer-character-sum}
\chi_{\hat{M} \tensor_{\hat{R}} k_v}(c)
  = \sum_{i:cGw_i = Gw_i} \chi_{B_{Gw_i}}(c) \cdot \chi_{U}(g^{-1}_{c,w_i}).
\end{equation}
\end{LEMMA}
\begin{proof}
The assertions will be derived by
passing to the quotient by $\mm_{\phi(v)}$ 
from the analogous statement for the $k(C)$-module structures on the 
completed rings $\hat{A}, \hat{B},$ etc.

Note that the Chinese Remainder Theorem isomorphism 
$$
\hat A \cong  \prod_{w \in \Phi_v} \hat A_w
$$
translates the $C$-action on $\hat{A}$ to a $C$-action by isomorphisms
$\hat{A}_w \overset{c}{\rightarrow} \hat{A}_{c(w)}$ permuting the factors
on the right. From this, and the isomorphisms
\[
\displaylines
{
\hat{B} \cong \prod_{i=1}^r \hat{B}_{Gw_i} \cr
\hat A \cong \hat B(G) \cong  \prod_{i=1}^r \hat B_{Gw_i}(G),\cr
}
\]
it is straightforward to check the assertions for the case
when $c \in C$ has $cGw_i=Gw_j$ for $j \neq i$.

If $c \in C$ has $cGw_i=Gw_i$, then there is an automorphism $c$ of 
$\hat{A}_{Gw_i}=\prod_{g \in G} \hat{A}_{gw_i} $,
which acts by isomorphisms 
$\hat{A}_{gw_i} \overset{c}{\rightarrow} \hat{A}_{gg_{c,w_i}w_i}$
between the components.
This action of $c$ translates over to 
$\hat{A}_{w_i}(G)$ using the isomorphism $\alpha$ from
equation \eqref{group-algebra-isomorphisms} in
the proof Lemma~\ref{regular-G-orbits}
which sends  $a t_g \overset{\alpha}{\mapsto} g(a) e_g$.
So $c$ sends $g(a) e_g$ to $cg(a) e_{g g_{c,w_i}}$,
and this maps under $\alpha^{-1}$ to
$$
(g g_{c,w_i})^{-1} cg(a) t_{g g_{c,w_i}} 
= g_{c,w_i}^{-1} g^{-1}cg(a) t_{g g_{c,w_i}} 
=  g_{c,w_i}^{-1} c(a) t_{g g_{c,w_i}}
\comma
$$
where the last equality uses the fact that $C$ and $G$ commute within $GL(V)$.
Thus $c(a t_g) =  g_{c,w_i}^{-1} c(a) t_{g g_{c,w_i}}$.  

From here one can translate the $C$-action to
$\hat{B}_{Gw_i} \cong \prod_{g \in G}\hat{A}_{gw_i}$ 
using the diagonal embedding 
$\hat{A}_{w_i} \overset{\Delta}{\rightarrow}  \hat{B}_{Gw_i}$
that maps $a \mapsto \Delta(a)=(g(a))_{g \in G}$.  One then finds that 
$g_{c,w_i}^{-1} c(a) \overset{\Delta}{\mapsto} c( \Delta(a))$.
In other words, a typical element $b t_g$ in $\hat{B}_{Gw_i}(G)$
is sent by $c$ to $c(b) t_{g g_{c,w_i}}$, as claimed.

The assertion about Brauer characters is then a consequence of
Lemma \ref{right-translation-character-lemma} which follows.
\end{proof}

\begin{LEMMA}
\label{right-translation-character-lemma}
Given an element $g_0$ in the finite group $G$ and a finite-dimensional
$k(G)$-module $U$,
let a cyclic group $C=\langle c \rangle$ act on $U$ via $c(u):=g_0^{-1}(u)$,
and
let $C$ act on $(k(G) \otimes_k U)^G$ via
$$
c\left(\sum_{g \in G} t_g \otimes u_g\right)
	:= \sum_{g \in G} t_{g g_0} \otimes u.
$$

Then $(k(G) \otimes_k U)^G \cong U$ as $k(C)$-modules, and consequently,
if $g_0$ is $p$-regular,
the Brauer character value for $c$ acting on $(k(G) \otimes_k U)^G$ is
$$
\chi_{(k(G) \otimes U)^G}(c) = \chi_U(g_0^{-1}).
$$
\end{LEMMA}
\begin{proof}
The map
$$
\begin{aligned}
U &\longrightarrow& (k(G) \otimes_k U)^G \\
u & \mapsto       & \sum_{g \in G} t_g \otimes g(u)
\end{aligned}
$$
is easily checked to give the desired $k(C)$-module isomorphism.
\end{proof}

\begin{proofof}[Theorem~\ref{PolynomialInvariantsTheorem}]
Assume that $k[V]^G$ is polynomial, and $c$ in $G$ is a regular element,
so that $c(v) = \omega v$ for some vector $v$ whose $G$-orbit $Gv$ is free.

We wish to apply Theorem~\ref{maintheorem} with $R=k[V]^G$,
so that $\Phi_v$ consists of only the regular $G$-orbit $Gv$
(that is, $r=1$ and $w_1=v$ is the representative of the unique $G$-orbit on
$\Phi_v$). In this case, the local rings $\hat{R}_{\phi(v)}=\hat{B}_{Gv}$
are the same, and their quotient $B(Gv)$ by the maximal ideal
$\mm_{\phi(v)}=\mm_{Gv}$ is the field $k_v\cong k$.
Thus as $k(G)$-modules, one has $A(\Phi_v) \cong k(G)$, and 
Theorem~\ref{maintheorem} implies that 
$$
\sum_{i \geq 0}(-1)^i \sum_{j \geq 0} \left[ \Tor^R_i(M, k)_j \right]
	= [(k(G) \otimes_k U)^G] = [U]
$$
as $k(\Gamma)$-modules, where $M=(U \otimes_k k[V])^G$,
for any finite-dimensional $k(G)$-module $U$ and finite subgroup $\Gamma$ of
$\Aut_{k(G)} U$.

We wish to also take into account the action of a cyclic group
$C=\langle \tau \rangle
\subset \Aut_{k(G)} V$ whose generator $\tau = c^{-1}$ acts as the scalar
$\omega^{-1}$ on $V$.
Then $\tau$ scales $V^*$ by $\omega$, and acts on the graded rings and modules
$R=k[V]^G, k[V], M:=(U \otimes k[V])^G$ by the scalar $\omega^j$ in their
$j^{th}$ homogeneous component, exactly as in the $C$-action considered in
Theorem~\ref{PolynomialInvariantsTheorem}.
Note that $\tau$ acts on $\Phi_v=Gv$ by $\tau(v)=c^{-1}(v)$ and more generally
$\tau(g(v))=\omega \cdot g(v) = gc^{-1}(v)$.
Thus Lemma~\ref{add-in-cyclic-action}
shows that the $k(C)$-structure on $A(\Phi_v) \cong k(G)$ has
$\tau(t_g)=t_{gc^{-1}}$,
in agreement with the $k(G \times C)$-structure on $k(G)$ that appeared
in Springer's Theorem, and  Lemma~\ref{right-translation-character-lemma}
then shows that the $k(\Gamma \times C)$ structure on $U$ agrees with the one
that appears in Theorem~\ref{PolynomialInvariantsTheorem}.
\end{proofof}

\subsection{Induced Modules and the Proof of
	Corollary~\ref{CyclicSievingCorollary}}
\label{induced-module-section}

We would like to apply Theorems~\ref{PolynomialInvariantsTheorem} and
\ref{maintheorem}
to modules $M$ which are (relative) invariants for a subgroup $H$ of $G$,
rather than
for $G$ itself; in particular, we would like to draw conclusions about the
$H$-invariant subring $k[V]^H$, as in Corollary~\ref{CyclicSievingCorollary}.

This is made possible by
a suitable notion of induction of $k(H)$-modules to $k(G)$-modules.
Regard $k(G)$ as a $k(G \times H)$-module via the action
$$
(g,h)\cdot t_{g'}:=gg'h^{-1}.
$$  
Then given any finite-dimensional $k(H)$-module $W$, define its
{\emphasis induced} $k(G)$-module to be
$$
\begin{aligned}
\Ind_H^G (W) &:= \Hom_{k(H)}(k(G),W)\\
           &= \{f \in \Hom_k (k(G),W): f(t_{gh^{-1}})=h(f(t_g)) \}.\\
\end{aligned}
$$
That this is a $k(G)$-module follows from the equality
$(g\cdot f)(t_{g'})=f(t_{g^{-1}g'})$.
We next explain how this construction converts relative invariants for $G$ into
relative invariants for $H$.

\begin{PROP} [Frobenius Reciprocity]
\label{Frobenius-reciprocity}
For finite-dimensional $k(G)$-modules $V$ and $k(H)$-modules $W$,
$$
(V \otimes_k \Ind_H^G (W))^G 
\cong 
(\Res^G_H (V) \otimes_k W)^H.
$$
\end{PROP}
\begin{proof} This is a standard isomorphism which follows 
from \cite[\sex3.3, Lemma 1]{Sefive} or
\cite[3.3.1, 3.3.2]{BensonRepBook}, giving
$$
(V \otimes_k \Ind_H^G (W))^G 
\cong
(\Ind_H^G(\Res^G_H (V) \otimes_k W))^G
\cong 
(\Res^G_H (V) \otimes_k W)^H.
$$
as required.\end{proof}

We will be particularly interested in the situation where the $k(H)$-module $W$
is the trivial module, so that $\Ind_H^G (W)=\Ind_H^G (k)$ is the permutation
module for $G$ on the cosets
$H \backslash G$, and the $H$-relative invariants are simply
the $H$-invariants, namely, $(V \otimes_k W)^H = (V \otimes_k k)^H=V^H$.  
In this situation, the isomorphism in Proposition~\ref{Frobenius-reciprocity}  
becomes a $k(\Gamma)$-module isomorphism for the action of the group
$\Gamma:=N_G(H)/H$. This is a special case of the following result.

\begin{PROP}
Consider nested subgroups $H \subset L \subset G$ of a finite group $G$,
and assume that $W=\Res^L_H (\widetilde{W})$ for some finite-dimensional
$k(L)$-module $\widetilde{W}$. Then 
\begin{enumerate}
\item[(i)] The group $\Gamma:=N_L(H)/H$ acts on $U:=\Ind_{H}^{G} (W)$ via
$$
(\gamma  \cdot f)(t_g):=\gamma(f(t_{g\gamma}))
$$
for $\gamma \in N_L(H).$

\item[(ii)] For $\widetilde{W} \neq 0$ this action is faithful, that is,
it injects $\Gamma$ into $\Aut_{k(G)}(U).$

\item[(iii)] The isomorphism
$$
(V \otimes_k \Ind_H^G (W))^G
\cong
(\Res^G_H (V) \otimes_k W)^H.
$$
of Proposition~\ref{Frobenius-reciprocity} is also a $k(\Gamma)$-isomorphism,
assuming one lets $k(\Gamma)$ act as follows:
  \begin{enumerate}
   \item[$\bullet$] on the left, solely in the factor $\Ind_H^G (W)$, while
   \item[$\bullet$] on the right, diagonally in both factors of
$\Res^G_H (V) \otimes_k W$.
  \end{enumerate}
\end{enumerate}
\end{PROP}

\begin{proof}
It is straightforward to check that the definition in (i) describes an action
of $N_L(H)$ on $\Ind_H^G (W)$ with $H$ acting trivially, and that this
action commutes with the $k(G)$-module structure.

To see that $\widetilde{W} \neq 0$ implies the action is faithful,
pick any $\tilde{w} \neq 0$ in $\widetilde{W}$ and
consider the element $f$ of $\Ind_H^G (W)$ defined by
$$
f(t_g) = 
\begin{cases}
g^{-1}(w) & \text{ if }g \in H,\\
0 & \text{ otherwise.}
\end{cases}
$$
An element $\gamma$ in $N_L(H)$ sends $f$ to the function 
$$
(\gamma \cdot f)(t_g) = 
\begin{cases}
\gamma ((g \gamma)^{-1}(w)) = g^{-1}(w) & \text{ if }g\gamma \in H,\\
0 & \text{ otherwise.}
\end{cases}
$$
Hence $\gamma \cdot f = f$ if and only if $\gamma$ lies in $H$.

It is also not hard to check that the isomorphism
given in the proof of Proposition~\ref{Frobenius-reciprocity}
is $k(\Gamma)$-equivariant for the actions described.
\end{proof}

\bigskip

We next provide the proof of our main application,
Corollary \ref{CyclicSievingCorollary},
resolving in the affirmative both Conjecture 3 and Question
4 in \cite{RSW}.

\begin{proofof} [Corollary \ref{CyclicSievingCorollary}]
Recall that the ground field $k$ is arbitrary  and
$H \subset G$ are two nested finite subgroups of $GL(V)$
where $V$ is a finite dimensional $k$-vector space.
We have assumed that the ring of invariants $k[V]^G$ is a polynomial algebra,
so it is its own Noether normalization and may be taken as $R$ in the results
established in this section. The group
$C=\langle c \rangle$ is a  cyclic subgroup generated by a regular element
$c$ in $G$, with eigenvalue $\omega$ on some regular eigenvector in $V$. 
We set $\Gamma:=N_G(H)/H$ and
have given $k[V]^H$ the $k(\Gamma \times C)$-structure in which $c$ scales the
variables in $V^*$ by $\omega$, and $\Gamma$ acts by linear substitutions.

What we must prove is that
in $\Groth(k(\Gamma \times C))$
\begin{equation}
\label{sieving-equivalent-equation}
\sum_{i \geq 0}(-1)^i \sum_{j \geq 0} \left[ \Tor^{k[V]^G}_i(k[V]^H,k)_j \right]
= \left[ k(H \backslash G) \right].
% \leqno{(\star)}
\end{equation}
So ignoring the action of $\Gamma$ this will imply that 
$$
X(t) =\frac{\left[k[V]^H\right](t)}{\left[k[V]^G\right](t)} \in \ZZ[[t]],
$$
is a polynomial in $t$, and the action of $C$ on the set
$X=H \backslash G$ then
gives a triple $(X,X(t),C)$ that exhibits the
cyclic sieving phenomenon of \cite{RStantonWhite}. Specifically,
for any element $c^j$ in $C$, the cardinality of its fixed point subset
$X^{c^j} \subset X$ is
given by evaluating $X(t)$ at a complex root-of-unity $\hat\omega^j$ of the
same multiplicative order as $c^j$, viz.,
$
|X^{c^j}| = \left[ X(t) \right]_{t=\hat{\omega}^j}.
$

In this context, we apply Theorem~\ref{PolynomialInvariantsTheorem}, together
with the results of this section, taking
$\widetilde{W}=k$ to be the trivial $k(G)$-module.
Then $W:=\Res^G_H (\widetilde{W}) = k$ is the trivial $k(H)$-module,
and 
$$
U:=\Ind_H^G (W)=\Ind_H^G k = k(H \backslash G)
$$
where $k(H \backslash G)$ carries the $k(\Gamma \times C)$-module structure
as described in the statement of Corollary \ref{CyclicSievingCorollary}.
So we have
$$
M=(U \otimes_k k[V])^G \cong (k\otimes_k k[V])^H= k[V]^H
\comma
$$
and unraveling the notations,
Theorem \ref{PolynomialInvariantsTheorem} tells us that
equality \eqref{sieving-equivalent-equation} holds in $\Groth(k(\Gamma \times C))$
as required. If we
ignore the $\Gamma$-action and compare Brauer character values for each
element $c^{-j}$ in $C$ on the two sides of equality \eqref{sieving-equivalent-equation}
we find, on the left side
$$
\sum_{i \geq 0}(-1)^i \left[ \Tor^{k[V]^G}(k[V]^H,k)\right](t)
	\Bigg\res_{t=\hat\omega^j}
	= \frac{\left[k[V]^H\right](t)}{\left[k[V]^G\right](t)}
	\Bigg\res_{t=\hat\omega^j}
= X(\hat\omega^j)
\comma
$$
where $\hat\omega$ is the Brauer lift
to $\CC^\times$ of $\omega \in k^\times$, while on the right
it is
the number of fixed points for $c^{j}$ permuting the elements of the
set $X=H\backslash G$.
\end{proofof}

\section{Character Values and Hilbert series}
\label{cyclic-only-section}

  We provide here further context and applications of
Theorem \ref{maintheorem}, modelled on certain variations of Molien's
theorem, which we first recall.

\subsection{Molien's theorem and Brauer character values}
\label{Molien-section}

Molien's theorem is most usually stated in characteristic zero, but it is well
known that it works equally in positive characteristic, using Brauer characters
instead of ordinary characters of representations.  We recall a version of this
result that works in all characteristics. In our statement we suppose that a
certain $k(G)$-module $U$ is projective. Observe that in characteristic zero this is no
restriction because all modules are projective. In positive characteristic $p$
the Brauer character of a $k(G)$-module is only defined on the $p$-regular
elements of $G$, that is, the elements whose order is prime to $p$.
So that our statement works in all characteristics, we describe these elements
as the ones whose order is invertible in $k$. The theorem has to do with the
Hilbert series for the module of $U$-relative
invariants $(U \otimes_k k[V])^G$, where $V$ is a $k(G)$-module, and by a
standard isomorphism this module of relative invariants is isomorphic to
$\Hom_{k(G)}(U^*,k[V])$. When $U=P_S$ is the projective cover of a simple
$k(G)$-module $S$ (in characteristic 0 this means $U\cong S$) the Hilbert
series of this module is thus the composition factor multiplicity series
$$
\sum_{n=0}^\infty [k[V]_n:S^*]\dim{\rm End}_{k(G)}(S) t^n.
$$

\begin{PROP}(Molien; see \cite{Mitchell} \sex1 Proposition 1.2 and
formula (1.5))
Let $V$ be a finite dimensional $k(G)$-module and $U$ a finite dimensional projective $k(G)$-module with Brauer character values $\chi_U(g)$. Define $\chi_U(g)=0$ when the order of $g$ is not invertible in $k$. Then one has
$$
\Hilb( (U \otimes_k k[V])^G, t) = 
\frac{1}{|G|} \sum_{ g \in G} \frac{\chi_U(g)}{\det(1-t g^{-1})}. 
$$
\end{PROP}

\noindent
The next result is an immediate corollary of Molien's theorem, appearing 
implicitly in the discussion of Springer \cite[(4.5)]{Springer}.
It describes a situation where information flows in the other
direction, that is, one can compute Brauer character values for $k(G)$-modules
from knowledge
of the Hilbert series of their relative invariants.  To state it,
given a graded algebra $R$ and graded
$R$-module $M$, as in Section~\ref{rationality-section}, define
$$
X_{M,R}(t)=\frac{\Hilb(M,t)}{\Hilb(R,t)}
	= \sum_{i \geq 0}(-1)^i \Hilb(\Tor_i^R(M,k),t).
$$
where the second equality comes from
Proposition~\ref{rationality-proposition}(\romannumeral\ref{seriesinvertible}).

\begin{CORY}
\label{nonmodular-Molien-corollary}
Let $V$ be a finite dimensional $k(G)$-module and $U$ a finite dimensional
projective $k(G)$-module with Brauer character $\chi_U$. 
Let $g\in G$ be such that the order of $g$ is invertible in $k$ and let
$\omega$ be an element of $k^\times$
such that the multiplicity of $\omega$ as an eigenvalue of elements of
$G$ is achieved uniquely by the conjugacy class of $g$.
Let $\hat\omega$ denote a Brauer lift of $\omega$.

Then the value of the Brauer character $\chi_U$ at $g$ is given by
\begin{equation}
\label{Springer-implicit-equation}
\chi_U(g) = X_{M, R}(\hat\omega),
\end{equation}
where $M:=(U \otimes_k k[V])^G$ is considered as a module over the 
invariant ring $R:=k[V]^G$.
\end{CORY}
\begin{proof}
Let $n:=\dim_k V$, and let $g$ have eigenvalues $\omega_1,\ldots,\omega_n$ on $V$,
with $\omega_1=\cdots=\omega_m=\omega$.  If $c$ denotes the cardinality of the
conjugacy class of $g$ in $G$, and if $d:=\prod_{j=m+1}^n (1-\hat\omega^{-1}\hat\omega_j)$,
then Molien's theorem implies that the Laurent expansions about $t=\hat\omega$ for the Hilbert
series of $M=(U \otimes_k k[V])^G$ and $R=k[V]^G$ begin similarly:
$$
\begin{aligned}
\Hilb(M,t) &= \frac{c \cdot \chi_U(g)}{|G|\cdot d \cdot (1-t\hat\omega^{-1})^m} 
  + O\left( \frac{1}{(1-t\hat\omega^{-1})^{m-1}}\right) \\ 
\Hilb(R,t) &= \frac{c}{|G|\cdot d \cdot (1-t\hat\omega^{-1})^m} 
  + O\left( \frac{1}{(1-t\hat\omega^{-1})^{m-1}}\right).
\end{aligned}
$$
Consequently, their quotient $X_{M,R}(t)$ has no pole at $t=\hat\omega$, 
and its value at $t=\hat\omega$ is $\chi_U(g)$.
\end{proof}

\noindent
We point out two old and one new application of Corollary~
\ref{nonmodular-Molien-corollary}.

\begin{EXAMPLE}
\label{Springer-original-example}
Springer's original theorem \cite[(4.5)]{Springer} 
asserts that the complex character values of a regular element $g$
in a finite complex reflection group always obey equation
\eqref{Springer-implicit-equation}.
In his proof he verified that the hypotheses of
Corollary~\ref{nonmodular-Molien-corollary} hold in this situation,
that is, the conjugacy class of a regular element uniquely 
achieves the maximum eigenvalue multiplicity for any of its eigenvalues corresponding to regular eigenvectors.

This should be compared with the special case of
Theorem~\ref{PolynomialInvariantsTheorem}
in which one forgets the $\Gamma$-action, and looks only at the Brauer
character values
for the $C$-action: it says exactly that, {\it for arbitrary modules} $U$, when
$g$ is a regular element of a group $G$ that has $k[V]^G$ polynomial, 
equation \eqref{Springer-implicit-equation} still holds for the Brauer
character value $\chi_U(g)$.
\end{EXAMPLE}

\begin{EXAMPLE}
Springer \cite[Assertion 2.2(iii)]{Springer} observed that
the conclusion of
Corollary~\ref{nonmodular-Molien-corollary}
applies to every element $g$ in the {\it binary icosahedral group},
the largest of the three sporadic finite subgroups of $SL_2(\CC)$.
Every element $g$
in this group has two eigenvalues $\omega, \omega^{-1}$, and these eigenvalues
happen to determine the conjugacy class of $g$ uniquely.
\end{EXAMPLE}

\begin{EXAMPLE}
\label{simple-group-168}
We show Corollary~\ref{nonmodular-Molien-corollary} applies 
to the simple group $G$ of order $168$, via its realization 
as the index two subgroup $G_{24} \cap SL_3(\CC)$ inside the complex reflection group $G_{24}$ 
from the list of Shephard and Todd \cite{ShephardTodd}; 
see Springer \cite[\S 4.6]{Springer} for details on this
realization. Letting $\alpha,i,\beta$
denote primitive complex roots-of-unity of orders $3,4,7$, respectively,
the six conjugacy classes within $G$ have these eigenvalues, orders of elements, and sizes:
\[
\begin{tabular}{|c|c|c|}\hline
\text{ eigenvalues } & \text{ order } & \text{ size of conjugacy class }
\tablestrut\\ \hline\hline
$(1,1,1)$	& 1 & 1 \tablestrut\\ \hline
$(1,-1,-1)$	& 2 & 21 \tablestrut\\ \hline
$(1,i,-i)$	& 4 & 42 \tablestrut\\ \hline
$(1,\alpha ,\alpha^2)$	& 3 & 56 \tablestrut\\ \hline
$(\beta,\beta^2,\beta^4)$	& 7 & 24 \tablestrut\\ \hline
$(\beta^3,\beta^5,\beta^6)$ & 7 & 24 \tablestrut\\ \hline
\end{tabular}
\]
\centerline{{\bf Table \ref{Molien-section}.1:} Conjugacy classes in the simple group of order $168$}
\def\ConjugacyClassTable{\ref{group-of-order-168-section}.1}

\noindent
Scrutiny of the first column of this table shows that every $g$
in $G$ has at least one eigenvalue
$\omega$ to which Corollary~\ref{nonmodular-Molien-corollary} applies.
Hence every $g$ in $G$ has its complex character values $\chi_U(g)$ satisfying
equation~\eqref{Springer-implicit-equation}.
\end{EXAMPLE}

\subsection{A modular version}
As with Example~\ref{Springer-original-example} and
Theorem~\ref{PolynomialInvariantsTheorem},
one would sometimes like to deduce results like
Corollary~\ref{nonmodular-Molien-corollary} for the Brauer character values
$\chi_U(g)$ of all modules $U$ in positive characteristic, not just the
projective modules. We derive here one such general result from
Theorem~\ref{maintheorem}, and then apply it
to the $3$-modular reduction of the group $G$ of order $168$
considered in Example~\ref{simple-group-168}.

\begin{PROP}
\label{cyclic-only-proposition}
Let $G$ be a finite subgroup of $GL(V)$, and $g$ a regular element of
$G$ with regular eigenvector $v$ and associated eigenvalue $\omega$.
Let $U$ be a $k(G)$-module and denote its Brauer character by
$\chi_U:G \rightarrow \CC$.
Let $\hat\omega \in \CC^\times$ be a Brauer lift of $\omega$.

Assume there exists a graded Noether normalization $R \subset k[V]^G$
having the following properties.\footnote{These hypotheses would hold
if $k[V]^G$ were polynomial and one took $R=k[V]^G$, as in Theorem~\ref{PolynomialInvariantsTheorem},
but we {\it do not} assume this here.}
\begin{enumerate}\advance\itemindent by \parindent
\item[(i)] $G$ acts freely on the fiber $\Phi_v=\phi^{-1}(\phi(v))$ containing
$v$ for the finite map $V \overset{\phi}{\rightarrow} \Spec R$. 
\item[(ii)] For every $G$-orbit $Gw_i$ in $\Phi_v$ that has
$\omega Gw_i=Gw_i$, the unique element
$g_{\omega,w_i} \in G$ for which 
%PW 29July08 Subscripts i inserted.
$\omega w_i = g_{\omega,w_i} w_i$ has the same
Brauer character value $\chi_U(g_{\omega,w_i})=\chi_U(g)$.
\item[(iii)] $X_{k[V]^G,R}(\hat\omega) \neq 0$.
\end{enumerate}

Then
$$
\chi_U(g) = X_{M, k[V]^G}(\hat\omega).
$$
where $M=(U \otimes_k k[V])^G$ is considered as a module over $k[V]^G$ and over $R$.
\end{PROP}

Observe that the hypotheses imply $g_{\omega^{-1},w_i}=g_{\omega,w_i}^{-1}$, hence $\chi_U(g_{\omega^{-1},w_i}^{-1})=\chi_U(g)$.

\begin{proof}
First note that Proposition~\ref{rationality-proposition}(iv) implies
$$
X_{M,R}(t) = X_{M,k[V]^G}(t) \cdot X_{k[V]^G,R}(t)
$$
and hence
\begin{equation}
\label{first-cancellation-equation}
X_{M,R}(\hat\omega) = X_{M,k[V]^G}(\hat\omega) \cdot X_{k[V]^G,R}(\hat\omega).
\end{equation}

Consider the cyclic group $C \subset \Aut_{k(G)} V$
whose generator $g$ acts as the scalar $\omega$ on $V$, 
and let $\tau = g^{-1}$
as in the proof of Theorem~\ref{PolynomialInvariantsTheorem}. 
Then $\tau$ scales $V^*$ by $\omega$, and acts on the graded rings and modules
$$
R \subset k[V]^G, \,\, k[V], \,\, M:=(U \otimes k[V])^G
$$ 
by the scalar $\omega^j$ on their
$j^{th}$ homogeneous component. Hypothesis (i) allows us to apply
Lemma \xref{add-in-cyclic-action}.

Therefore 
\begin{equation}
\label{second-cancellation-equation}
\begin{aligned}
X_{M,R}(\hat\omega) &= \sum_{i \geq 0} (-1)^i\chi_{\Tor^R_i(M,k)}(\tau)\\
	&= \chi_{\hat{M} \tensor_{\hat{R}} k_v}(\tau)\\
	& = \sum_{i:\;\omega Gw_i =Gw_i}
		\chi_{B_{Gw_i}}(\tau) \cdot \chi_U(g_{\tau,w_i}^{-1})\\
	&= \chi_U(g) \cdot \sum_{i:\;\omega Gw_i =Gw_i} \chi_{B_{Gw_i}}(\tau) 
\end{aligned}
\end{equation}
where the first equality uses the definition of the scalar action on $R,M$\/;
\linebreak[4]
the second equality is a consequence of Lemma \xref{regular-G-orbits} (iii);
substituting $c^{-1}$ for $\tau$
(to conform to the notations of the proof of Lemma \xref{add-in-cyclic-action})
and remembering $\tau$ scales by $\omega$ on $V^*$
the third equality becomes
equation \eqref{Brauer-character-sum} of that lemma;
and the last equality uses hypothesis (ii)
above and the observation made before the proof,
bearing in mind that $g$ and $\tau$ act inversely
on $v$.

In particular, taking $U=k$ the trivial $k(G)$-module in
\eqref{second-cancellation-equation} gives
\begin{equation}
\label{third-cancellation-equation}
X_{k[V]^G,R}(\hat\omega) = \sum_{i:\;\omega Gw_i =Gw_i} \chi_{B_{Gw_i}}(\tau).
\end{equation}
Putting together equations \eqref{first-cancellation-equation}, 
\eqref{second-cancellation-equation}, and 
\eqref{third-cancellation-equation},
one obtains
\begin{equation}
\label{triple-equality}
X_{M,k[V]^G}(\hat\omega) \cdot X_{k[V]^G,R}(\hat\omega)
	= X_{M,R}(\hat\omega)
=\chi_U(g) \cdot X_{k[V]^G,R}(\hat\omega)
\end{equation}
Since by hypothesis (iii) 
$X_{k[V]^G,R}(\hat\omega)$ is a {\it nonzero} factor of
the extreme left and right terms of this string of equalities,
we may divide by it yielding the desired equality.
\end{proof}

\EXAMPLE
\label{modular-simple-group-168}
We apply Proposition~\ref{cyclic-only-proposition} to
the $3$-modular reduction of the simple group of order $168$
that was considered in
Example 3 of section~\ref{Molien-section}
This application also highlights the flexibility of using various different
Noether normalizations $R$ inside $k[V]^G$.

As preparation, we consider the $3$-modular reduction of
the complex reflection group $G_{24}$ considered by
Shephard and Todd \cite{ShephardTodd}.  Kemper and Malle \cite[\S 6, p. 74]{KemperMalle} describe a
realization of it as a subgroup (which we also denote $G_{24}$) 
of $GL_3(\FF_9)$ having order $336$.  Their realization
is generated by three reflections, each of order two, and each unitary
with respect to the nondegenerate sesquilinear form on $V:=(\FF_9)^3$
defined by $(x,y)=\sum_{i=1}^3 x_i \bar{y}_i$ where $\bar{y}:=y^3$ denotes
the Frobenius automorphism of $\FF_9$.  They also exhibit
explicitly three $G_{24}$-invariant polynomials $f_4, f_6, f_{14}$ of degrees
$4,6,14$ respectively, with the property that their Jacobian determinant
$f_{21}$ is nonvanishing.  Since 
$$
4 \cdot 6 \cdot 14 = 336 = |G_{24}|
$$
a criterion of Kemper \cite[Theorem 3.7.5]{DerksenKemper} implies that 
$$
k[V]^{G_{24}}=k[f_4,f_6,f_{14}]
$$ 
is a polynomial algebra for any extension $k$ of $\FF_9$.  

The subgroup $G:=G_{24} \cap SL_3(\FF_9)$ has index two in $G$, and is
isomorphic to the simple group of order $168$.  Since $G_{24}$ contains
the scalar transformation
$\tau:=-1_{V}$, one has $G_{24}=G \times \langle \tau \rangle$.

To analyze the ring of $G$-invariants, we first analyze
the module of $\det$-relative invariants for $G_{24}$, that is,
$$
k[V]^{G_{24},\det}:=(\det \otimes k[V])^{G_{24}}
$$
where $\det: G_{24} \rightarrow \{\pm 1\}$
denotes the determinant character of $G_{24}$.
Because $G_{24}$ is generated by 
involutive reflections, none of which are transvections,
the Jacobian determinant $f_{21}$ is the product of the linear forms defining
the reflecting hyperplanes for
the $21$ reflections in $G_{24}$;
see Broer \cite{Broer} or Hartmann and Shepler \cite{HartmannShepler}.

One also knows  that $f_{21}$ is a $\det$-relative invariant
for $G$.  Conversely, any $\det$-relative
invariant is divisible by each linear form defining one of the reflecting
hyperplanes, and so must also be divisible by $f_{24}$. Thus 
$$
k[V]^{G_{24},\det} = f_{21} k[V]^{G_{24}}
$$ 
is a free $k[V]^{G_{24}}$-module of rank one.

We claim that this implies the following decomposition of the
$G$-invariant ring
\begin{equation}
\label{odd-even-decomposition}
\begin{aligned}
k[V]^G_{even} &= k[V]^{G_{24}}\\
k[V]^G_{odd} &= k[V]^{G_{24}, \det}.
\end{aligned}
\end{equation}
where $k[V]^G_{even}, k[V]^G_{odd}$ denote the subspaces spanned by
homogeneous $G$-invariants of even, odd degree.
To see this, note that any of the generating involutive reflections
$\sigma$ for $G_{24}$ has $\sigma \tau$ lying in the subgroup $G$.
Hence any homogeneous element of $f$ in $k[V]^G$,
say of of degree $d$, will be fixed by $\sigma\tau$ and thus satisfy
$$
\sigma(f)=\tau(f)=(-1)^d f=
\begin{cases}
f & \text{ if }d\text{ is even},\\
\det(\sigma) f & \text{ if }d\text{ is odd}.
\end{cases}
$$  
Since these reflections $\sigma$ generate $G_{24}$, the decomposition in
\eqref{odd-even-decomposition} follows.
Consequently, 
$$
\begin{aligned}
k[V]^G & = k[f_4,f_6,f_{14}] \,\, \oplus \,\, f_{21}\cdot k[f_4,f_6,f_{14}],
	\text{ and } \\
\Hilb(k[V]^G,t) & = \frac{1+t^{21}}{(1-t^4)(1-t^6)(1-t^{14})}.
\end{aligned}
$$
We wish to use this information to deduce the analogue of
Example 1 in section \ref{Molien-section}
for all $3$-modular Brauer characters of $G$.
Assume that the extension $k$ of $\FF_9$
is algebraically closed.  The conjugacy classes of $G$ are described
just as in Table \ref{Molien-section}.1, except that in this case one must 
\begin{enumerate}
\item[$\bullet$]
re-interpret $i, \beta$ as roots-of-unity of orders $4, 7$ in $k^\times$,
\item[$\bullet$]
re-interpret the cube root-of-unity $\alpha$ as $1$,
since the conjugacy class of elements of order $3$ are not $3$-regular and act unipotently.
\end{enumerate}

\begin{PROP}
Let $G \subset SL_3(\FF_9)$ be the $3$-modular reduction of the
simple group of order $168$.  Let $k$ be an algebraically closed extension 
of $\FF_9$, and for any $3$-regular element $g$ of $G$, let
$\omega$ in $k^\times$ be any eigenvalue for $g$ having the same multiplicative order as $g$. 

Then for any $k(G)$-module $U$ one has the Brauer character value
$$
\chi_U(g)=X_{M , k[V]^G}(\hat\omega)
$$
where $M:=(U \otimes k[V])^G$, and $\hat\omega$ is the Brauer lift of $\omega$.
\end{PROP}

\begin{proof}
We wish to apply Proposition~\ref{cyclic-only-proposition} with $v$ chosen
generically from the $\omega$-eigenspace for $g$. We will use two different
Noether normalizations $R$ inside $k[V]^G$, depending upon the order of
the $3$-regular element $g$. 

For elements $g$ of order $2$, so that $\omega=-1$, we use
the subring
$$
R=k[f_4,f_6,f_{21}]
$$
which we claim is a Noether normalization.
To check this claim, one must show that the only solution to the system
\begin{equation}
\label{hsop-system}
f_4(v)=f_6(v)=f_{21}(v)=0
\end{equation}
is the vector $v=0$.  For this, we borrow the idea of Springer \cite[3.2(i)]{Springer}.  Given
any solution $v$ to \eqref{hsop-system}, 
scaling $v$ by a $14^{th}$ root-of-unity $\xi$ in $k^\times$ gives a vector
$\xi v$ with the property that $f(\xi v)=f(v)$ for $f=f_4,f_6,f_{21}$, as well as for $f=f_{14}$.  
Hence $f(\xi v)=f(v)$ for every $f$ in $k[V]^G$, so that $\xi v$ and $v$ must
represent the same point in the quotient space $V/G$.  Thus $g(v)=\xi v$ for some $g$ in $G$.  
But if $v \neq 0$, then such an element $g$ would have order divisible by $14$, 
and there are no such elements of $G$.  Hence $v=0$.  

For this choice of $R$, one has
$$
X_{k[V]^G,R}(t)= 1+t^{14}+t^{28}
$$
so that $X_{k[V]^G,R}(-1)=3 \neq 0$.  Hence 
hypothesis (iii) of Proposition~\ref{cyclic-only-proposition} will be
satisfied. We claim that hypothesis (i) of Proposition~\ref{cyclic-only-proposition} will be
satisfied as long as $v$ is chosen generically from the $2$-dimensional
$(-1)$-eigenspace for $g$ inside $V$.  To see this, note that  $\Phi(v)=\phi^{-1}(\phi(v))$ 
will carry a free $G$-action as long as $v$ avoids the union
$$
Y:=\bigcup_{\substack{h \in G:\\ h \neq 1}} \phi^{-1}(\phi(V^h)).
$$
For each $h \neq 1$ in $G$, the
fixed subspace $V^h$ is of dimension at most $1$.  
Since $\phi: V \rightarrow \Spec R$ is a finite morphism, each set $\phi^{-1}(\phi(V^h))$
is contained in a $1$-dimensional algebraic subset of $V=k^3$.
Since the union $Y$ is finite, 
$Y$ is also contained in a $1$-dimensional algebraic set.
Thus the vector $v$ chosen generically inside the $2$-dimensional $(-1)$-eigenspace will
indeed avoid $Y$, since $k$ is algebraically closed (and hence infinite).
Lastly, hypothesis (ii) of Proposition~\ref{cyclic-only-proposition} will 
be satisfied: all elements $g_{\omega,w_i}$ for
the various orbits $Gw_i$ in $\Phi_v$ have the same order $2$ as $\omega$,
and all elements of $G$ of order $2$ are conjugate, so they are all conjugate to $g$.

For elements $g$ of order $1,4,$ or $7$, we use the Noether normalization
$$
\begin{aligned}
R=k[V]^{G_{24}}&=k[f_4,f_6,f_{14}]\\
X_{k[V]^G,R}(t)&= 1+t^{21}.
\end{aligned}
$$
Note that $X_{k[V]^G,R}(\hat\omega) \neq 0$ for such elements $g$,
so hypothesis (iii) of Proposition~\ref{cyclic-only-proposition} will be
satisfied.  

Since $G_{24}=G \times \langle \tau \rangle$,
the fiber $\Phi_v$ decomposes into the two $G$-orbits 
$$
\Phi_v= G_{24}v = Gv \sqcup -Gv.
$$
To show that $G$ acts freely on $\Phi_v$, and thereby satisfies hypothesis 
(i) of Proposition~\ref{cyclic-only-proposition}, we consider three cases
depending upon the order of $g$.  

If $g$ has order $1$, since $v$ is chosen generically within all of $V$,
there is no problem.  

If $g$ has order $7$, note that $f_4(v)=f_6(v)=0$
since any $f$ in $k[V]^G$ which is homogeneous of degree $d$ will have
$$
f(v)=f(g(v))=f(\omega v)=\omega^d f(v).
$$
If in addition, $\Phi_v=G_{24}v$ does not carry a free $G$-action, then it carries
a non-free $G_{24}$-action, and a theorem of Serre~\cite{Serre-poly-invariants} implies that 
$v$ lies on some reflecting hyperplane for $G_{24}$, that is,
$f_{21}(v)=0$.  Hence $v$ is a solution to equation \eqref{hsop-system}, so $v=0$,
a contradiction.

If $g$ has order $4$, we argue directly why the isotropy
subgroup 
$$
G_v:=\{h\in G: h(v)=v\}
$$ 
must be trivial.  Firstly, there can be no elements
$h$ in $G_v$ order $7$, as such elements in $G$ have no eigenvalues
equal to $1$.  Secondly, if $G_v$ contained an element of {\it even} order, then without
loss of generality it would contain some $h$ of order $2$.  Unitarity of $h$ would
imply that $h$ acts on the $2$-dimensional perpendicular space $v^\perp$ as the scalar
$-1$, while unitarity of $g$ implies that $g$ acts on $v^\perp$ with its other eigenvalues $1, -\omega$.  
Thus $g$ and $h$ would commute, and $gh$ would have eigenvalues $\omega$ (on $v$) and 
$-1, \omega$ (on $v^\perp$).  As there are no elements of $G$ with these eigenvalues, 
$G_v$ can contain no elements of even order.  Lastly, this forces $G_v$ to be a $3$-torsion
group, and since $|G|$ is not divisible by $9$, either $G_v$ is trivial or $G_v=\{1,h,h^{-1}\}$ 
for some $h$ of order $3$.  But the latter cannot occur:  when
$\langle g \rangle=\{1,g,g^2,g^3\}$  acts on $G_v$ by conjugation, one finds that
$ghg^{-1}=h^{\pm 1}$, so that $g^2$ commutes with $h$ and $g^2h$ would be an element of order
$6$ in $G$.

Finally, note that since $\Phi_v$ contains only the two $G$-orbits
$Gv, -Gv$, one has $g_{\omega, -v}=g_{\omega,v}=g$, and hence hypothesis (ii)
of Proposition \ref{cyclic-only-proposition} is always satisfied.
\end{proof}

\end{document}